\documentclass[a4paper,10pt, oneside]{article}
\usepackage{amsmath,amssymb,amsthm,bm,mathrsfs,graphicx,enumerate}
\usepackage{authblk}
\usepackage[font=small,labelfont=md,textfont=it]{caption}
\usepackage{floatrow,float,listings}
\usepackage[titletoc, title]{appendix}
\usepackage[colorlinks,linkcolor=blue,citecolor=blue]{hyperref}
\usepackage{etoolbox}
\usepackage{longtable}
\usepackage{diagbox}
\usepackage{booktabs,makecell,multirow}
\usepackage[capitalise,sort]{cleveref}
\usepackage{cases,color}
\crefname{equation}{}{}
\crefname{lemma}{Lemma}{Lemmas}
\crefname{theorem}{Theorem}{Theorems}
\crefname{discr}{Discretization}{Discretizations}

\DeclareMathOperator{\D}{D}

\apptocmd{\sloppy}{\hbadness 10000\relax}{}{}

\newcommand{\dual}[1]{\langle {#1} \rangle}

\newcommand{\Dual}[1]{\left\langle {#1} \right\rangle}

\newcommand{\nm}[1]{\lVert {#1} \rVert}
\newcommand{\Nm}[1]{\left\lVert {#1} \right\rVert}

\newcommand{\snm}[1]{\lvert {#1} \rvert}

\newcommand{\ssnm}[1]
{
  \left\vert\kern-0.25ex
  \left\vert\kern-0.25ex
  \left\vert
  {#1}
  \right\vert\kern-0.25ex
  \right\vert\kern-0.25ex
  \right\vert
}

\makeatletter
\def\spher@harm#1{%
  \vbox{\hbox{%
    \offinterlineskip
    \valign{&\hb@xt@2\p@{\hss$##$\hss}\vskip.2ex\cr#1\crcr}%
  }\vskip-.36ex}%
}
\def\gshone{\spher@harm{.}}
\def\gshtwo{\spher@harm{.&.}}
\def\gshthree{\spher@harm{.&.&.}}
\let\gsh\spher@harm
\makeatother

\newtheorem{lemma}{Lemma}[section]
\newtheorem{remark}{Remark}[section]
\newtheorem{theorem}{Theorem}[section]

\makeatletter\def\@captype{table}\makeatother

\begin{document}
\title{
  \Large\bf Temporally semidiscrete approximation of a Dirichlet boundary
  control for a fractional/normal evolution equation with a final observation
  \thanks{ This work was supported by National Natural Science Foundation of
  China (11901410).
  }
}

\author{
  {\sc Qin Zhou\thanks{Email:zqmath@aliyun.com}
  and
  Binjie Li\thanks{Corresponding author. Email: libinjie@scu.edu.cn}
  } \\
  School of Mathematics, Sichuan University.
}

\date{}
\maketitle

\begin{abstract} 
  %
  Optimal Dirichlet boundary control for a fractional/normal evolution with a
  final observation is considered. The unique existence of the solution and the
  first-order optimality condition of the optimal control problem are derived.
  The convergence of a temporally semidiscrete approximation is rigorously
  established, where the control is not explicitly discretized and the state
  equation is discretized by a discontinuous Galerkin method in time. Numerical
  results are provided to verify the theoretical results.

\end{abstract}

\section{Introduction}

There is an extensive literature on the numerical optimization with PDE
constraints. So far, most of the literature focuses on the distributed control
problems, and the works on the Dirichlet boundary control problems are rather
limited. Compared with the distributed control problems, the Dirichlet boundary
control problems are more challenging in the following senses. Firstly, the
solution of the state equation of a Dirichlet boundary control problem is of
significantly lower regularity than that of a distributed control problem, and
this increases the difficulty in both theoretical and numerical analysis.
Secondly, in the weak form of the solution of the state equation of a Dirichlet
boundary control problem, the test function space is more regular than the trial
function space, and hence the weak form is not appropriate for the
discretization. Thirdly, since the normal derivative of the adjoint state occurs
in the first-order optimality condition, the discrete first-order optimality
condition will essentially involve the discrete normal derivative of the
discrete adjoint state, and this increases the implementation difficulty.

We summarize the works on the parabolic Dirichlet boundary control problems
briefly as follows. Using an integral representation formula derived by the
semigroup theory (cf.~\cite[Section~4.12]{Balakrishnan1981} and
\cite{Lasiecka1980}), Lasiecka \cite{Lasiecka1980-optim,Lasiecka1984} analyzed
spatial Galerkin approximations of an optimal Dirichlet boundary control problem
and a time optimal Dirichlet boundary control problem for the parabolic
equations. Kunisch and Vexler \cite{Kunisch2007} analyzed constrained Dirichlet
boundary control problems for a class of parabolic equations and derived the
convergence of the PDAS strategy for two Dirichlet boundary control problems.
Applying the Robin penalization method to a Dirichlet boundary control problem
for a parabolic equation with a final observation, Belgacem et
al.~\cite{Belgacem2011} obtained a penalized Robin boundary control problem.
Gong et~al.~\cite{Gong2016} analyzed a finite element approximation of a
Dirichlet boundary control for a parabolic equation, where the variational
discretization approach \cite{Hinze2005} was used and the state equation was
discretized by the usual $ H^1 $-conforming $ P1 $-element in space and
discretized by the $\text{dG}(0)$ scheme in time. Recently, Gong and Li
\cite{Gong-Li-2019} improved the spatial accuracy derived in \cite{Gong2016},
using the maximal $ L^p $-regularity theory. We note that, for the state
equation with rough Dirichlet boundary data,
\cite{Lasiecka1980-optim,Lasiecka1984} used the semigroup theory to define the
solution whereas \cite{Gong2016,Gong-Li-2019,Kunisch2007} used the transposition
technique to define the solution (called the very weak solution).


For the numerical analysis of parabolic Neumann/Robin boundary control problems,
we refer the reader to
\cite{Walter1989,Chrysafinos2014,Knowles1982,Malanowski1981}. For the numerical
analysis of other optimal control problems for parabolic equations, we refer the
reader to
\cite{Deckelnick2011,Gong2014,Vexler2013,Vexler2016a,Meidner2007,Vexler2008I,Vexler2008II,Meidner2011}
and the references therein. Although the spatial discretization is not
considered in this paper, we would like to refer the reader to
\cite{Berggren2004,French1991,French1993,Lasiecka1986} for the numerical
analysis of elliptic and parabolic equations with rough Dirichlet boundary data.


To our best knowledge, no convergence result is available for the Galerkin-type
approximations of the Dirichlet boundary control problems governed by the
parabolic equations with final observations. The fractional evolution equation
is an extension of the normal evolution equation, widely used to describe the
physical phenomena with memory effect \cite{Podlubny1998}. Recently, Harbir et
al.~\cite{Harbir2016} studied an optimal distributed control problem for a
space-time fractional diffusion equation. For the numerical analysis of the
optimal distributed problems governed by the time fractional diffusion
equations, we refer the reader to
\cite{Gunzburger2019,Jin_optim_2020,Li-Xie-Yan2020,Zhang2019}. To our knowledge,
no numerical analysis is available for the Dirichlet boundary control problems
governed by the fractional evolution equations. Hence, this paper tries to
analyze the Dirichlet boundary control problems for the fractional and normal
evolution equations in a unified way.

In this paper, we establish the convergence of a temporally semidiscrete
approximation of an abstract optimal control problem governed by a
fractional/normal evolution equation with a final observation. This
approximation uses the variational discretization concept \cite{Hinze2005} and
uses a discontinuous Galerkin method to discretize the state equation in time.
The discontinuous Galerkin method is the famous $\text{dG}(0)$ scheme for the
normal evolution equation, and is equivalent to the well-known L1 scheme
\cite{Lin2007} with uniform temporal grids for the fractional evolution
equations.
The derived numerical analysis is applied to a Dirichlet boundary control
problem.
We note that there are many works (see
\cite{Jin2015IMA,Jin2013SIAM,Jin2016-L1,Jin2016,Jin2018-time-dependtn,Li2019SIAM,Lubich1996,Luo2019}
and the references therein) devoted to the numerical analysis of the fractional
diffusion equations with rough initial value and source term, but, to our
knowledge, no numerical analysis is available for the fractional diffusion
equation with rough Dirichlet boundary value. This paper also fills in this gap.
The rest of this paper is organized as follows. \cref{sec:abs_optim} establishes
the convergence of a temporally semidiscrete approximation of an abstract
optimal control problem. \cref{sec:appl} applies the theory developed in the
previous section to a Dirichlet boundary control problem. \cref{sec:numer}
performs three numerical experiments to confirm the theoretical results.

\section{An abstract optimal control problem}
\label{sec:abs_optim}
\subsection{Preliminaries}
\label{sec:pre}
We will use the following conventions: for each linear vector space, the field
of the scalars is $ \mathbb C $; for a Hilbert space $ \mathcal X $, we use $
(\cdot,\cdot)_\mathcal X $ to denote its inner product; for a Banach space $
\mathcal B $, we use $ \dual{\cdot,\cdot}_{\mathcal B} $ to denote a duality
paring between $ \mathcal B^* $ (the dual space of $ \mathcal B $) and $
\mathcal B $; for a linear operator $ A $, $ \rho(A) $ denotes the resolvent set
of $ A $ and $ R(z,A) $ denotes the inverse of $ z-A $ for each $ z \in \rho(A)
$; for two Banach spaces $ \mathcal B_1 $ and $ \mathcal B_2 $, $ \mathcal
L(\mathcal B_1, \mathcal B_2) $ is the set of all bounded linear operators from
$ \mathcal B_1 $ to $ \mathcal B_2 $, and $ \mathcal L(\mathcal B_1, \mathcal
B_1) $ is abbreviated to $ \mathcal L(\mathcal B_1) $; $ I $ denotes the
identity map; for a Lebesgue measurable subset $ \mathcal D \subset \mathbb R^l
$, $ 1 \leqslant l \leqslant 4 $, $ \dual{p, q}_{\mathcal D} $ means the
integral $ \int_{\mathcal D} p \overline q $, where $ \overline q $ is the
conjugate of $ q $; for a function $ v $ defined on $ (0,T) $, by $ v(t-) $, $ 0
< t \leqslant T $, we mean the limit $ \lim_{s \to t-} v(s) $; the notation $
C_\times $ means a positive constant, depending only on its subscript(s), and
its value may differ at each occurrence; for any $ 0 < \theta < \pi $, define
\begin{align} 
  \Sigma_\theta &:= \{
    re^{i\gamma}: \, r > 0,
    -\theta < \gamma < \theta
  \}, \label{eq:Sigma_theta-def} \\
  \Gamma_\theta &:= \{re^{-i\theta}: r \geqslant 0 \} \cup \{re^{i\theta}: r >0\}
  \label{eq:Upsilon-def} \\
  \Upsilon_\theta &:= \{
    z \in \Gamma_\theta:\
    -\pi \leqslant \Im z \leqslant \pi
  \}, \label{eq:Upsilon1-def}
\end{align}
where $ i $ is the imaginary unit and $ \Gamma_\theta $ and $ \Upsilon_\theta $
are so oriented that the negative real axis is to their left.

\subsubsection{Time fractional Sobolev spaces} 
Assume that $ -\infty < a < b < \infty $ and $ \mathcal B $ is a Banach space.
Define
\begin{align*}
  {}_0H^1(a,b;\mathcal B) &:= \left\{
    v \in L^2(a,b;\mathcal B):\,
    v' \in L^2(a,b;\mathcal B), \, v(a) = 0
  \right\}, \\
  {}^0H^1(a,b;\mathcal B) &:= \left\{
    v \in L^2(a,b;\mathcal B):\,
    v' \in L^2(a,b;\mathcal B), \, v(b) = 0
  \right\},
\end{align*}
and endow them with the two norms
\begin{align*}
  \nm{v}_{{}_0H^1(a,b;\mathcal B)} &:=
  \nm{v'}_{L^2(a,b;\mathcal B)}
  \quad\forall v \in {}_0H^1(a,b;\mathcal B), \\
  \nm{v}_{{}^0H^1(a,b;\mathcal B)} &:=
  \nm{v'}_{L^2(a,b;\mathcal B)}
  \quad\forall v \in {}^0H^1(a,b;\mathcal B),
\end{align*}
respectively, where $ v' $ is the first-order weak derivative of $ v $.

For each $ 0 < \gamma < 1 $, define
\begin{align*}
  {}_0H^\gamma(a,b;\mathcal B) &:=
  (L^2(a,b;\mathcal B),\, {}_0H^1(a,b;\mathcal B))_{\gamma,2}, \\
  {}^0H^\gamma(a,b;\mathcal B) &:=
  (L^2(a,b;\mathcal B),\, {}^0H^1(a,b;\mathcal B))_{\gamma,2},
\end{align*}
where $ (\cdot,\cdot)_{\gamma,2} $ means the interpolation space defined by the
$ K $-method (cf.~\cite{Lunardi2018}). For convenience, the spaces $
{}_0H^\gamma(a,b;\mathbb C) $ and $ {}^0H^\gamma(a,b;\mathbb C) $ will be
abbreviated to $ {}_0H^\gamma(a,b) $ and $ {}^0H^\gamma(a,b) $, respectively.

\subsubsection{Riemann-Liouville fractional calculus operators}
Assume that $ -\infty < a < b < \infty $ and $ \mathcal X $ is a separable
Hilbert space. For any $ 0 < \gamma < 1 $, define
\begin{align*}
  \left( \D_{a+}^{-\gamma} v\right)(t) &:=
  \frac1{ \Gamma(\gamma) }
  \int_a^t (t-s)^{\gamma-1} v(s) \, \mathrm{d}s,
  \quad \text{a.e.}~t \in (a,b), \\
  \left(\D_{b-}^{-\gamma} v\right)(t) &:=
  \frac1{ \Gamma(\gamma) }
  \int_t^b (s-t)^{\gamma-1} v(s) \, \mathrm{d}s,
  \quad\text{a.e.}~t \in (a,b),
\end{align*}
for all $ v \in L^1(a,b;\mathcal X) $, where $ \Gamma(\cdot) $ is the gamma
function. In addition, let $ \D_{a+}^0 $ and $ \D_{b-}^0 $ be the identity
operator on $ L^1(a,b;\mathcal X) $. For any $ 0 < \gamma \leqslant 1 $, define
\begin{align*}
  \D_{a+}^\gamma v & := \D \, \D_{a+}^{\gamma-1}v, \\
  \D_{b-}^\gamma v & := -\D \, \D_{b-}^{\gamma-1}v,
\end{align*}
for all $ v \in L^1(a,b;\mathcal X) $, where $ \D $ is the first-order
differential operator in the distribution sense.

Assume that $ 0 < \gamma < 1 $. For any $ v \in {}_0H^\gamma(a,b;\mathcal X) $
and $ w \in {}^0H^\gamma(a,b;\mathcal X) $, we have
\begin{align*}
  C_1 \nm{v}_{{}_0H^\gamma(a,b;\mathcal X)} & \leqslant
  \nm{\D_{a+}^\gamma v}_{L^2(a,b;\mathcal X)} \leqslant
  C_2 \nm{v}_{{}_0H^\gamma(a,b;\mathcal X)}, \\
  C_1 \nm{w}_{{}^0H^\gamma(a,b;\mathcal X)} & \leqslant
  \nm{\D_{b-}^\gamma w}_{L^2(a,b;\mathcal X)} \leqslant
  C_2 \nm{w}_{{}^0H^\gamma(a,b;\mathcal X)},
\end{align*}
where $ C_1 $ and $ C_2 $ are two positive constants depending only on $ \gamma
$. Let $ \mathcal X^* $ be the dual space of $ \mathcal X $. For any $ v \in
{}_0H^{\gamma/2}(a,b;\mathcal X^*) $ and $ w \in {}^0H^{\gamma/2}(a,b;\mathcal
X) $, the equality
\begin{equation}
  \label{eq:dual}
  \int_0^T \dual{\D_{a+}^\gamma v, w}_{\mathcal X} \, \mathrm{d}t =
  \int_0^T \dual{v, \D_{b-}^\gamma w}_{\mathcal X} \, \mathrm{d}t
\end{equation}
holds for the following two cases: $ v \in {}_0H^\gamma(a,b;\mathcal X^*) $ and
$ w \in {}^0H^\gamma(a,b;\mathcal X) $; $ \D_{a+}^\gamma v \in
L^{2/(1+\gamma)}(a,b;\mathcal X^*) $ and $ \D_{b-}^\gamma v \in
L^{2/(1+\gamma)}(a,b;\mathcal X) $. For the above theoretical results, we refer
the reader to \cite{Ervin2006,Luo2019}.

\subsubsection{Definitions of \texorpdfstring{$ \mathcal A $ and $ \mathcal A^* $}{}}
Assume that $ X $ and $ Y $ are two separable Hilbert spaces such that $ X $ is
continuously embedded into $ Y $ and $ X $ is dense in $ Y $. We will regard $ Y
$ as a subspace of $ X^* $, the dual space of $ X $, in the sense that
\[
  \dual{v, w}_X := (v,w)_Y \quad
  \text{ for all $ v \in Y $ and $ w \in X $}.
\]
Let $ \mathcal A $ and $ \mathcal A^* $ be two bounded linear operators from $ X
$ to $ Y $ satisfying that
\begin{subequations}
\begin{numcases}{}
  \rho(\mathcal A) \supset \Sigma_{\omega_0} \cup \{0\},
  \quad \rho(\mathcal A^*) \supset \Sigma_{\omega_0} \cup \{0\},
  \label{eq:resolvent} \\
  \nm{R(z,\mathcal A)}_{\mathcal L(Y)} \leqslant
  \frac{\mathcal M_0}{1+\snm{z}} \quad
  \forall z \in \Sigma_{\omega_0}, \label{eq:R(z,A)-Y} \\
  \nm{R(z,\mathcal A^*)}_{\mathcal L(Y)} \leqslant
  \frac{\mathcal M_0}{1+\snm{z}}
  \quad \forall z \in \Sigma_{\omega_0},
  \label{eq:R(z,A*)-Y} \\
  (\mathcal Av, w)_Y = (v, \mathcal A^* w)_Y \quad \forall v, w \in X,
  \label{eq:A-dual} \\
  c_0 \nm{v}_X \leqslant \nm{\mathcal Av}_Y \leqslant
  c_1 \nm{v}_X \quad \forall v \in X, \label{eq:A-X} \\
  c_0 \nm{v}_X \leqslant \nm{\mathcal A^*v}_Y \leqslant
  c_1 \nm{v}_X \quad \forall v \in X, \label{eq:A*-X} \\
  (\mathcal Av, v)_Y \geqslant 0 \quad \forall v \in X,
  \label{eq:A-positive}
\end{numcases}
\end{subequations}
where $ \pi/2 < \omega_0 < \pi $, $ c_0 $, $ c_1 $ and $ \mathcal M_0 $ are four
positive constants. By the transposition technique, $ \mathcal{A} $ and $
\mathcal{A^*} $ can be extended as two bounded linear operators from $ Y $ to $
X^* $ by
\begin{align}
  \dual{\mathcal{A}v, w}_X &:= (v, \mathcal A^* w)_Y, \label{eq:A-ext} \\
  \dual{\mathcal{A^*} v,w}_X &:= (v, \mathcal A w)_Y, \label{eq:A*-ext}
\end{align}
for all $ v \in Y $ and $ w \in X $.

For each $ 0 \leqslant \theta \leqslant 1 $, let $ [X^*,Y]_\theta $ and $
[X,Y]_\theta $ be the interpolation spaces defined by the famous complex
interpolation method (cf.~\cite[Chapter 2]{Lunardi2018}). We have that $
[X^*,Y]_\theta $ is the dual space of $ [X,Y]_\theta $ and vice versa
(cf.~\cite{Calderon1964}). By \eqref{eq:A-X}, \eqref{eq:A*-X} and \cite[Theorem
2.6]{Lunardi2018}, a straightforward computation gives that
\begin{equation}
  \label{eq:A}
  \nm{\mathcal A}_{
    \mathcal L( [X,Y]_\theta, [X^*,Y]_{1-\theta} )
  } \leqslant c_1 \quad \text{for all } 0 \leqslant \theta \leqslant 1.
\end{equation}
\begin{lemma} 
  \label{lem:z-A}
  Assume that $ 0 \leqslant \theta \leqslant 1 $ and $ z \in \Sigma_{\omega_0}
  $. Then
  \begin{align}
    \nm{R(z,\mathcal A)}_{\mathcal L(Y,[X,Y]_\theta)}
    \leqslant \frac{C_{c_0,\mathcal M_0}}{1+\snm{z}^\theta},
    \label{eq:A-Y-X-Y} \\
    \nm{R(z,\mathcal A^*)}_{\mathcal L(Y,[X,Y]_\theta)}
    \leqslant \frac{C_{c_0,\mathcal M_0}}{1+\snm{z}^\theta},
    \label{eq:A*-Y-X-Y} \\
    \nm{R(z,\mathcal A)}_{\mathcal L([X^*,Y]_{\theta},Y)}
    \leqslant \frac{C_{c_0,\mathcal M_0}}{1+\snm{z}^\theta},
    \label{eq:A-X*-Y-Y} \\
    \nm{R(z,\mathcal A^*)}_{\mathcal L([X^*,Y]_{\theta},Y)}
    \leqslant \frac{C_{c_0,\mathcal M_0}}{1+\snm{z}^\theta}.
    \label{eq:A*-X*-Y-Y}
  \end{align}
  Moreover, for any $ 0 \leqslant \epsilon \leqslant 1 $,
  \begin{equation}
    \label{eq:1000}
    \nm{R(z,\mathcal A)}_{
      \mathcal L( [X^*,Y]_\theta, [X,Y]_{1-(1-\epsilon)\theta} )
    } \leqslant \frac{
      C_{c_0,\mathcal M_0,\epsilon,\theta}
    }{1+\snm{z}^{\epsilon\theta}}.
  \end{equation}
\end{lemma}
\begin{proof} 
  A straightforward computation gives
  \begin{align*}
    \nm{\mathcal AR(z,\mathcal
    A)}_{\mathcal L(Y)} = \nm{zR(z,\mathcal A) - I}_{\mathcal L(Y)} \leqslant
    \snm{z} \nm{R(z,\mathcal A)}_{\mathcal L(Y)} + 1 \leqslant 1 + \mathcal M_0,
  \end{align*}
  by (\ref{eq:R(z,A)-Y}), so that (\ref{eq:A-X}) implies
  \[
    \nm{R(z,\mathcal A)}_{\mathcal L(Y,X)} \leqslant \frac{1+\mathcal M_0}{c_0}.
  \]
  By this estimate and (\ref{eq:R(z,A)-Y}), using \cite[Theorem
  2.6]{Lunardi2018} yields \cref{eq:A-Y-X-Y}. Estimate \cref{eq:A*-Y-X-Y} can be
  proved analogously.

  Then let us prove \cref{eq:A-X*-Y-Y}. We first consider the following problem:
  seek $ v \in Y $ such that
  \begin{equation}
    \label{eq:731} (v,(\overline z-\mathcal
    A^*)\varphi)_Y = \dual{g,\varphi}_X \quad \forall \varphi \in X,
  \end{equation}
  where $ g \in X^* $ is arbitrary but fixed. By \eqref{eq:A*-X} and the fact
  that $ X $ is continuously embedded into $ Y $, we conclude that $
  (\cdot,(\overline z - \mathcal A^*)\cdot)_Y $ is a continuous bilinear form on
  $ Y \times X $. Inserting $ \theta=0 $ into \cref{eq:A*-Y-X-Y} implies that,
  for any $ v \in Y \setminus \{0\} $,
  \[
    \sup_{\varphi \in X \setminus \{0\}} \frac{
      \snm{(v,(\overline z - \mathcal A^*)\varphi)_Y}
    }{\nm{\varphi}_X} \geqslant \frac{ \nm{v}_Y^2 }{
      \nm{R(\overline z,\mathcal A^*)v}_X
    } \geqslant C_{c_0,\mathcal M_0} \nm{v}_Y.
  \]
  Since $ z \in \rho(\mathcal A^*) $, it is evident that, for any $ \varphi \in
  X \setminus \{0\} $,
  \[
    \sup_{v \in Y} \snm{(v, (\overline z - \mathcal A^*) \varphi)_Y} > 0.
  \]
  Consequently, the Babuska-Lax-Milgram theorem yields that problem
  \cref{eq:731} admits a unique solution $ v \in Y $ and $ \nm{v}_Y \leqslant
  C_{c_0,\mathcal M_0} \nm{g}_{X^*} $. Since \cref{eq:A-ext,eq:731} imply $ v =
  R(z,\mathcal A) g $, this indicates that
  \[
    \nm{R(z,A)g}_Y \leqslant C_{c_0,\mathcal M_0} \nm{g}_{X^*} \quad
    \forall g \in X^*,
  \]
  and hence
  \begin{equation}
    \label{eq:R(z,A)X*}
    \nm{R(z,A)}_{\mathcal L(X^*,Y)} \leqslant C_{c_0,\mathcal M_0}.
  \end{equation}
  By \eqref{eq:R(z,A)-Y} and \eqref{eq:R(z,A)X*}, using \cite[Theorem
  2.6]{Lunardi2018} yields \cref{eq:A-X*-Y-Y}. Estimate \cref{eq:A*-X*-Y-Y} is
  derived similarly.

  Finally, let us prove \cref{eq:1000}. Inserting $ \theta = \epsilon $ and $
  \theta = 0 $ into \cref{eq:A-Y-X-Y} and \cref{eq:A-X*-Y-Y} respectively yields
  \[
    \nm{R(z,\mathcal A)}_{
      \mathcal L(Y, [X,Y]_\epsilon)
    } \leqslant \frac{C_{c_0,\mathcal M_0}}{1+\snm{z}^\epsilon},
    \quad \nm{R(z,\mathcal A)}_{\mathcal L(X^*,Y)}
    \leqslant C_{c_0,\mathcal M_0}.
  \]
  Using \cite[Theorem 2.6]{Lunardi2018} then gives
  \begin{align*}
    \nm{R(z,\mathcal A)}_{
      \mathcal L(
      [X^*,Y]_\theta,
      [Y,[X,Y]_\epsilon]_\theta
      )
    } \leqslant \frac{C_{c_0,\mathcal M_0}}{1+\snm{z}^{\epsilon\theta}}.
  \end{align*}
  Hence, by the fact that (cf.~\cite{Calderon1964})
  \[
    [Y,[X,Y]_\epsilon]_\theta =
    [X,Y]_{1-(1-\epsilon)\theta}
    \quad\text{with equivalent norms},
  \]
  we readily obtain \cref{eq:1000}. This completes the proof.
\end{proof}
\begin{remark} 
  For any $ z \in \Sigma_{\omega_0} $,
  \begin{align*}
    \nm{R(z,\mathcal A)}_{\mathcal L(X^*)} &=
    \nm{z^{-1}(z-\mathcal A + \mathcal A)R(z,\mathcal A)}_{\mathcal L(X^*)} \\
    &= \nm{I + \mathcal A R(z,\mathcal A)}_{\mathcal L(X^*)} / \snm{z} \\
    & \leqslant \frac{
      1 + \nm{\mathcal A R(z,\mathcal A)}_{\mathcal L(X^*)}
    }{\snm{z}} \\
    & \leqslant \frac{
      1 + c_1\nm{R(z,\mathcal A)}_{\mathcal L(X^*,Y)}
    }{\snm{z}} \quad\text{(by inserting $ \theta = 1 $ into \cref{eq:A})} \\
    & \leqslant \frac{ C_{c_0,c_1,\mathcal M_0} }{ \snm{z}}
    \quad\text{(by inserting $ \theta = 0 $ into \cref{eq:A-X*-Y-Y}).}
  \end{align*}
  Also, we have
  \begin{align*}
    \nm{R(z,\mathcal A)}_{\mathcal L(X^*)} &=
    \nm{\mathcal A \mathcal A^{-1} R(z,\mathcal A)}_{\mathcal L(X^*)} \\
    & \leqslant c_1 \nm{\mathcal A^{-1} R(z,\mathcal A)}_{\mathcal L(X^*,Y)}
    \quad\text{(by inserting $\theta=1$ into \cref{eq:A})} \\
    & \leqslant C_{c_1,\mathcal M_0} \nm{R(z,\mathcal A)}_{\mathcal L(X^*,Y)}
    \quad\text{(by \eqref{eq:R(z,A)-Y})} \\
    & \leqslant C_{c_0,c_1,\mathcal M_0}
    \quad\text{(by inserting $\theta=0$ into \cref{eq:A-X*-Y-Y}).}
  \end{align*}
  Consequently,
  \begin{equation}
    \label{eq:A-X*}
    \nm{R(z,\mathcal A)}_{\mathcal L(X^*)}
    \leqslant \frac{C_{c_0,c_1,\mathcal M_0}}{1+\snm{z}}
    \quad\forall z \in \Sigma_{\omega_0}.
  \end{equation}
\end{remark}


\subsubsection{Definitions of \texorpdfstring{$ E_\alpha $ and $ E_\alpha^* $}{}}
For any $ 0 < \alpha \leqslant 1 $ and $ t > 0 $, define
\begin{align}
  E_\alpha(t) &:= \frac1{2\pi i} \int_{\Gamma_{\omega_0}}
  e^{tz} R(z^\alpha,\mathcal A) \, \mathrm{d}z,
  \label{eq:E-def} \\
  E^*_\alpha(t) &:= \frac1{2\pi i}
  \int_{\Gamma_{\omega_0}} e^{t\overline z}
  R(\overline{z^\alpha},\mathcal A^*) \,
  \overline{\mathrm{d}z}.
  \label{eq:E*-def}
\end{align}
For any $ v \in X^* $ and $ w \in Y $, by the definitions of $ \mathcal A $ and
$ \mathcal A^* $ and \cref{lem:z-A} we have that
\[
  (R(z,\mathcal A)v, w)_Y =
  \dual{v,R(\overline z, \mathcal A^*)w}_X
  \quad \forall z \in \Sigma_{\omega_0},
\]
and hence from \cref{eq:E-def,eq:E*-def} we obtain
\begin{equation}
  \label{eq:E-E*}
  (E_\alpha(t)v, w)_Y = \dual{v, E_\alpha^*(t)w}_X
  \quad \forall t > 0.
\end{equation}
Furthermore, by \cref{lem:z-A}, a routine calculation
(cf.~\cite{Thomee2006,Jin2016}) yields the following lemma.
\begin{lemma}
  \label{lem:E-E*}
  Assume that $ 0 < \alpha \leqslant 1 $, $ 0 \leqslant \theta \leqslant 1 $, $
  t > 0 $, and $ G = E_\alpha $ or $ E^*_\alpha $. Then
  \begin{align} 
    \nm{G(t)}_{ \mathcal L(Y,[X,Y]_\theta) } &
    \leqslant C_{c_0,\omega_0,\mathcal M_0}
    t^{\theta\alpha-1}, \label{eq:E-X-Y} \\
    \nm{G(t)}_{ \mathcal L([X^*,Y]_\theta, Y) } &
    \leqslant C_{c_0,\omega_0,\mathcal M_0}
    t^{\theta\alpha-1}, \label{eq:E-X*-Y} \\
    \nm{G'(t)}_{\mathcal L(Y,[X,Y]_\theta)} &
    \leqslant C_{c_0,\omega_0,\mathcal M_0}
    t^{\theta\alpha-2}, \label{eq:E'-X-Y} \\
    \nm{G'(t)}_{\mathcal L([X^*,Y]_\theta,Y)} &
    \leqslant C_{c_0,\omega_0,\mathcal M_0}
    t^{\theta\alpha-2}. \label{eq:E'-X*-Y}
  \end{align}
  Moreover, for any $ 0 \leqslant \epsilon \leqslant 1 $,
  \begin{equation}
    \label{eq:201}
    E_\alpha \in C((0,\infty); \mathcal
    L([X^*,Y]_\theta,[X,Y]_{1-(1-\epsilon)\theta})
  \end{equation}
  and
  \begin{equation}
    \label{eq:AE}
    \nm{
      E_\alpha(t)
    }_{
      \mathcal L([X^*,Y]_\theta, [X,Y]_{1-(1-\epsilon)\theta})
    } \leqslant C_{c_0,\omega_0,\mathcal M_0,\epsilon,\theta} \,
    t^{\epsilon\theta\alpha-1}.
  \end{equation}
\end{lemma}
\begin{remark}
  By \cref{eq:A-X*} we have
  \[
    \nm{E_\alpha(t)}_{\mathcal L(X^*)}
    \leqslant C_{c_0,c_1,\omega_0,\mathcal M_0}
    t^{\alpha-1}, \quad \forall t>0.
  \]
\end{remark}

\subsubsection{Solutions of the fractional evolution equations}
Following the mild solution theory of fractional/normal evolution equations in
\cite{Pazy1983,Lubich1996,McLean2010-B,Jin2016}, we introduce the following mild
solutions. Assume that $ 0 < \alpha \leqslant 1 $ and $ 0 < T < \infty $. For
any
\[
  g \in L^1(0,T;[X^*,Y]_\theta) \quad
  \text{with}\quad 0 \leqslant \theta \leqslant 1,
\]
we call
\begin{equation}
  \label{eq:Sg-l1}
  (\mathcal S_\alpha g)(t) := \int_0^t E_\alpha(t-s) g(s) \,
  \mathrm{d}s, \quad \text{a.e.}~0 < t \leqslant T,
\end{equation}
the mild solution to the following fractional evolution equation:
\begin{equation}
  \label{eq:linear}
  (\D_{0+}^\alpha - \mathcal A) w = g,
  \quad w(0) = 0.
\end{equation}
For any $ v \in [X^*,Y]_\theta $, $ 0 \leqslant \theta \leqslant 1 $, we call
\begin{equation}
  \label{eq:Sdelta}
  (\mathcal S_\alpha(v\delta_0))(t) :=
  E_\alpha(t) v, \quad 0 < t \leqslant T,
\end{equation}
the mild solution to \cref{eq:linear} with $ g = v\delta_0 $, where $ \delta_0 $
is the Dirac measure in time concentrated at $ t=0 $. Symmetrically, for any
\[
  g \in L^1(0,T;[X^*,Y]_\theta) \quad
  \text{with}\quad 0 \leqslant \theta \leqslant 1,
\]
we call
\begin{equation}
  \label{eq:S*g}
  (\mathcal S_\alpha^*g)(t) := \int_t^T E^*_\alpha(s-t) g(s) \,
  \mathrm{d}s, \quad \text{a.e.}~0 < t < T,
\end{equation}
the mild solution to the following backward fractional evolution equation:
\begin{equation}
  \label{eq:linear*}
  (\D_{T-}^\alpha - \mathcal A^*) w = g,
  \quad w(T) = 0.
\end{equation}
For any $ v \in [X^*,Y]_\theta $, $ 0 \leqslant \theta \leqslant 1 $, we call
\begin{equation}
  \label{eq:S*delta}
  (\mathcal S_\alpha^*(v\delta_T))(t) := E^*_\alpha(T-t)
  v, \quad 0 < t \leqslant T,
\end{equation}
the mild solution to equation \cref{eq:linear*} with $ g = v\delta_T $, where $
\delta_T $ is the Dirac measure in time concentrated at $ t=T $.


\begin{lemma}
  \label{lem:Sg-dual-weakly}
  Assume that $ 0 < \alpha, \theta \leqslant 1 $ and $ q > 1/(\theta\alpha) $.
  Then
  \begin{align}
    \mathcal S_\alpha & \in \mathcal L(
    L^2(0,T;[X^*,Y]_\theta), \, L^2(0,T;Y)
    ), \label{eq:Sg-l2} \\
    \mathcal S_\alpha & \in \mathcal L(
    L^q(0,T;[X^*,Y]_\theta), \, C([0,T];Y)
    ). \label{eq:Sg-C}
  \end{align}
  Moreover, for any $ g \in L^q(0,T;[X^*,Y]_\theta) $ and $ v \in Y $,
  \begin{equation}
    \label{eq:Sdelta-dual}
    ((\mathcal S_\alpha g)(T), v)_Y = \int_0^T \Dual{
      g(t), (\mathcal S_\alpha^*(v\delta_T))(t)
    }_{X} \, \mathrm{d}t.
  \end{equation}
\end{lemma}
\begin{proof} 
  By \cref{eq:E-X*-Y,eq:Sg-l1}, a routine argument (cf.~\cite[Theorem
  2.6]{Diethelm2010}) yields \cref{eq:Sg-l2,eq:Sg-C}. Note that
  \cref{eq:Sg-l1,eq:Sg-C} imply
  \[
    (\mathcal S_\alpha g)(T) =
    \int_0^T E_\alpha(T-t)g(t) \, \mathrm{d}t,
  \]
  and hence
  \begin{align*}
    ((\mathcal S_\alpha g)(T), v)_Y &=
    \left(
      \int_0^T E_\alpha(T-t) g(t) \, \mathrm{d}t, v
    \right)_Y = \int_0^T (E_\alpha(T-t) g(t) , v )_Y \, \mathrm{d}t \\
    &= \int_0^T \Dual{g(t), E_\alpha^*(T-t)v}_{X} \, \mathrm{d}t
    \quad\text{(by \cref{eq:E-E*})} \\
    &= \int_0^T \Dual{
      g(t), (\mathcal S_\alpha^*(v\delta_T))(t)
    }_{X} \, \mathrm{d}t
    \quad\text{(by \cref{eq:S*delta}).}
  \end{align*}
  This proves \cref{eq:Sdelta-dual} and hence this lemma.
\end{proof}

\begin{lemma}  
  \label{lem:S-to-1}
  For any $ 0 < \theta \leqslant 1 $,
  \begin{align}
    \lim_{\alpha \to {1-}} \nm{
      \mathcal S_\alpha - \mathcal S_1
    }_{
      \mathcal L(L^1(0,T;[X^*,Y]_\theta),
      \, L^1(0,T;Y))
    } = 0, \label{eq:S-to-1-l1} \\
    \lim_{\alpha \to {1-}} \nm{
      \mathcal S_\alpha - \mathcal S_1
    }_{
      \mathcal L(L^\infty(0,T;[X^*,Y]_\theta),
      \, C([0,T];Y))
    } = 0. \label{eq:S-to-1}
  \end{align}
\end{lemma}
\begin{proof} 
  Since
  \begin{align*}
    R(z,\mathcal A) - R(z^\alpha,\mathcal A) =
    (z^\alpha-z) R(z,\mathcal A) R(z^\alpha,\mathcal A)
    \quad\text{ for all } z \in \Sigma_{\omega_0},
  \end{align*}
  a straightforward calculation gives, by \cref{eq:E-def,lem:z-A}, that
  \begin{align*}
    \nm{ (E_\alpha - E_1)(t) }_{\mathcal L([X^*,Y]_\theta,Y)}
    \leqslant C_{c_0,\mathcal M_0} \int_0^\infty
    e^{t\cos\omega_0 r} \frac{
      \snm{
        re^{i\omega_0} - (re^{i\omega_0})^\alpha
      }
    }{
      (1+r)(1+r^{\theta\alpha})
    } \, \mathrm{d}r
  \end{align*}
  for all $ t > 0 $. It follows that
  \[
    \nm{
      E_\alpha - E_1
    }_{L^1(0,T;\mathcal L([X^*,Y]_\theta,Y))}
    \leqslant{}
    C_{c_0,\omega_0,\mathcal M_0}
    \int_0^\infty
    \frac{
      \snm{
        re^{i\omega_0} - (re^{i\omega_0})^\alpha
      }
    }{
      r(1+r)(1+r^{\theta\alpha})
    } \, \mathrm{d}r.
  \]
  Then Lebesgue's dominated convergence theorem yields
  \begin{equation}
    \label{eq:lbj}
    \lim_{\alpha \to {1-}} \nm{
      E_\alpha - E_1
    }_{L^1(0,T;\mathcal L([X^*,Y]_\theta,Y))} = 0.
  \end{equation}
  Since Young's inequality implies
  \[
    \nm{\mathcal S_\alpha - \mathcal S_1}_{
      \mathcal L(L^1(0,T;[X^*,Y]_\theta), L^1(0,T;Y)
    } \leqslant
    \nm{\mathcal E_\alpha - E_1}_{
      L^1(0,T;\mathcal L([X^*,Y]_\theta))
    },
  \]
  by \cref{eq:lbj} we readily obtain \cref{eq:S-to-1-l1}. Moreover, by
  \cref{eq:Sg-l1,eq:Sg-C} we have
  \[
    \nm{
      \mathcal S_\alpha - \mathcal S_1
    }_{
      \mathcal L(L^\infty(0,T;[X^*,Y]_\theta,Y), C([0,T];Y))
    } \leqslant
    \nm{
      E_\alpha - E_1
    }_{L^1(0,T;\mathcal L([X^*,Y]_\theta,Y))},
  \]
  so that \cref{eq:lbj} proves \cref{eq:S-to-1}. This completes the proof.
\end{proof}

\begin{lemma}
  \label{lem:Sg-regu}
  Assume that $ 0 < \alpha,\theta \leqslant 1 $. Then for any $ g \in
  C([0,T];[X^*,Y]_\theta) $ we have
  \begin{equation}
    \label{eq:Sg-strong}
    (\D_{0+}^\alpha - \mathcal A) \mathcal S_\alpha g = g
  \end{equation}
  and
  \begin{equation}
    \label{eq:Sg-regu}
    \begin{aligned}
      & \nm{\D_{0+}^\alpha \mathcal S_\alpha g}_{
        C([0,T];[X^*,Y]_{(1-\epsilon)\theta})
      } + \nm{
        \mathcal A \mathcal S_\alpha g
      }_{
        C([0,T];[X^*,Y]_{(1-\epsilon)\theta})
      } \\
      \leqslant{} &
      c \nm{g}_{C([0,T];[X^*,Y]_\theta)},
    \end{aligned}
  \end{equation}
  where $ 0 < \epsilon < 1 $ and $ c $ is a positive constant independent of $ g
  $.
\end{lemma}
\begin{proof}
  Since a complete rigorous proof of this lemma is tedious but standard
  (cf.~\cite{Pazy1983}), we only present briefly the key ingredients of the
  proof.

  {\it Step 1.} Define
  \begin{equation}
    \label{eq:eta}
    \eta(t) := \frac1{2\pi  i} \int_{\Gamma_{\omega_0}}
    e^{tz} z^{\alpha-1} R(z^\alpha,\mathcal A) \, \mathrm{d}z,
    \quad t > 0.
  \end{equation}
  A straightforward computation gives that, for any $ 0 < t \leqslant T $,
  \begin{align*}
    \eta(t) &= \frac1{2\pi i} \int_{\Gamma_{\omega_0}}
    e^{tz} z^{-1}(z^\alpha-\mathcal A + \mathcal A)
    R(z^\alpha,\mathcal A) \, \mathrm{d}z \\
    &= \frac1{2\pi i} \int_{\Gamma_{\omega_0}}
    e^{tz} z^{-1} I \, \mathrm{d}z +
    \frac1{2\pi i} \int_{\Gamma_{\omega_0}}
    e^{tz} z^{-1} \mathcal A
    R(z^\alpha,\mathcal A) \, \mathrm{d}z \\
    &= I + \frac1{2\pi i} \mathcal A
    \int_{\Gamma_{\omega_0}} e^{tz} z^{-1} R(z^\alpha,\mathcal A) \,
    \mathrm{d}z,
  \end{align*}
  where $ \Gamma_{\omega_0} $ is deformed so that the origin is to its left.
  Hence, we conclude from \cref{eq:A,lem:z-A} the following properties:
  \begin{align*}
    & \eta \in C([0,T];\mathcal L([X^*,Y]_\theta,X^*))
    \cap C^1((0,T];\mathcal L([X^*,Y]_\theta,X^*)); \\
    & \eta(0) = I; \\
    & \eta'(t) = \frac1{2\pi i} \mathcal A
    \int_{\Gamma_{\omega_0}} e^{tz} R(z^\alpha,\mathcal A) \, \mathrm{d}z,
    \quad t > 0; \\
    & \nm{\eta'(t)}_{\mathcal L([X^*,Y]_\theta,X^*)}
    \leqslant C_{c_0,c_1,\omega_0,\mathcal M_0}
    t^{\theta\alpha-1}, \quad t > 0.
  \end{align*}

  {\it Step 2.} By the theory of Laplace transform, from \cref{eq:Sg-l1} we
  obtain that
  \begin{equation}
    \label{eq:426}
    (\D_{0+}^{\alpha-1} \mathcal S_\alpha g)(t) =
    \int_0^t \eta(t-s) g(s) \, \mathrm{d}s,
    \quad 0 < t \leqslant T.
  \end{equation}
  Hence, by the properties of $ \eta $ presented in Step 1,
  \begin{align*}
    & (\D_{0+}^\alpha \mathcal S_\alpha g)(t) =
    \frac{\mathrm{d}}{\mathrm{d}t}
    (\D_{0+}^{\alpha-1} \mathcal S_\alpha g)(t) =
    \frac{\mathrm{d}}{\mathrm{d}t}
    \int_0^t \eta(t-s) g(s) \, \mathrm{d}s \\
    ={} & g(t) + \int_0^t \eta'(t-s) g(s) \, \mathrm{d}s \\
    ={} & g(t) + \mathcal A \int_0^t \frac1{2\pi i}
    \int_{\Gamma_{\omega_0}} e^{(t-s)z}
    R(z^\alpha,\mathcal A) \, \mathrm{d}z \, g(s) \, \mathrm{d}s \\
    ={} &
    g(t) + \mathcal A \int_0^t E_\alpha(t-s) g(s) \, \mathrm{d}s
    \quad\text{(by \cref{eq:E-def})} \\
    ={} &
    g(t) + \mathcal A (\mathcal S_\alpha g)(t)
    \quad\text{(by \cref{eq:Sg-l1})}
  \end{align*}
  for each $ 0 \leqslant t \leqslant T $. This implies equality
  \cref{eq:Sg-strong}.

  {\it Step 3.} For convenience, we will use $ c $ to denote a positive
  constant, whose value is independent of $ g $ but may differ at each
  occurrence. A routine calculation gives, by \cref{eq:Sg-l1,eq:201,eq:AE}, that
  \[
    \nm{\mathcal S_\alpha g}_{
      C([0,T];[X,Y]_{1-(1-\epsilon)\theta})
    } \leqslant c \nm{g}_{C([0,T];[X^*,Y]_\theta)},
  \]
  so that \cref{eq:A} implies
  \begin{equation}
    \label{eq:700}
    \nm{
      \mathcal A \mathcal S_\alpha g
    }_{
      C([0,T];[X^*,Y]_{(1-\epsilon)\theta})
    } \leqslant c \nm{g}_{C([0,T];[X^*,Y]_\theta)}.
  \end{equation}
  Since $ [X^*,Y]_\theta $ is continuously embedded into $
  [X^*,Y]_{(1-\epsilon)\theta} $, we have
  \begin{equation}
    \label{eq:701}
    \nm{g}_{C([0,T];[X^*,Y]_{(1-\epsilon)\theta})}
    \leqslant c \nm{g}_{C([0,T];[X^*,Y]_\theta)}.
  \end{equation}
  Combining \cref{eq:Sg-strong,eq:700,eq:701} proves \cref{eq:Sg-regu} and thus
  concludes the proof.
\end{proof}

\subsection{Continuous problem}
Let $ Z $ be a Hilbert space and let $ U_\text{ad} \subset L^\infty(0,T;Z) $ be
a convex, bounded and closed subset of $ L^2(0,T;Z) $. We consider the following
abstract optimal control problem:
\begin{equation}
  \label{eq:abs-optim}
  \min\limits_{u \in
  U_{\text{ad}}} J_\alpha(u) :=
  \frac12 \nm{(\mathcal S_\alpha\mathcal R_{\theta_0}u)(T) - y_d}_Y^2 + \frac\nu2
  \nm{u}_{L^2(0,T;Z)}^2,
\end{equation}
where $ 0 < \alpha \leqslant 1 $, $ y_d \in Y $, $ \nu > 0 $ is a regularization
parameter, and $ \mathcal R_{\theta_0}: Z \to [X^*,Y]_{\theta_0} $ is a bounded
linear operator for some $ 0 < \theta_0 \leqslant 1 $.

Define $ \mathcal R_{\theta_0}^*: [X,Y]_{\theta_0} \to Z $ by
\[
  (\mathcal R_{\theta_0}^* v, w)_Z :=
  \overline{
    \dual{\mathcal R_{\theta_0}w,v}_{[X,Y]_{\theta_0}}
  }
\]
for all $ v \in [X,Y]_{\theta_0} $ and $ w \in Z $. Assume that $ q >
\max\{1/(\theta_0\alpha),2\} $. By \cref{eq:Sg-C}, $ (\mathcal S_\alpha\mathcal
R_{\theta_0}\cdot)(T) $ is a bounded linear operator from $ L^q(0,T;Z) $ to $ Y
$. Clearly, $ J_\alpha $ in \cref{eq:abs-optim} is a strictly convex functional
on $ L^q(0,T;Z) $, and $ U_\text{ad} $ is a convex, bounded and closed subset of
$ L^q(0,T;Z) $. By \cref{eq:Sdelta-dual}, a routine argument (cf.~\cite[Theorems
2.14 and 2.21]{Troltzsh2010}) yields the following theorem.
\begin{theorem}
  \label{thm:basic-regu}
  Problem \cref{eq:abs-optim} admits a unique solution $ u \in U_\text{ad} $,
  and the following first-order optimality condition holds:
  \begin{subequations}
  \begin{numcases}{}
    y = \mathcal S_\alpha\mathcal R_{\theta_0}u, \label{eq:optim-y} \\
    p = \mathcal S_\alpha^*\big( (y(T)-y_d)\delta_T \big), \label{eq:optim-p} \\
    \int_0^T \big(
      \mathcal R_{\theta_0}^*p(t) + \nu u(t), \, v(t) - u(t)
    \big)_Z \, \mathrm{d}t \geqslant 0
    \text{ for all } v \in U_\text{ad}.
    \label{eq:optim-u}
  \end{numcases}
  \end{subequations}
\end{theorem}

\subsection{Temporally discrete problem}
\label{ssec:discrete}
Let $ J > 1 $ be an integer and define $ t_j := j\tau $ for each $ j=0, 1, 2,
\dots, J $, where $ \tau := T/J $. For each Banach space $ \mathcal X $, define
\[ 
  W_\tau(\mathcal X) := \{
    V \in L^\infty(0,T; \mathcal X):\, V
    \text{ is constant on } (t_{j-1},t_j) \quad \forall 1 \leqslant j \leqslant J
  \}.
\]
For any $ 0 < \alpha < 1 $ and $ g \in L^1(0,T;X^*) $, define $ \mathcal
S_{\alpha,\tau}g \in W_{\tau}(Y) $ and $ \mathcal S_{\alpha,\tau}^*g \in
W_{\tau}(Y) $, respectively, by that
\begin{align}
  \int_0^T \dual{
    (\D_{0+}^\alpha - \mathcal A) \mathcal S_{\alpha,\tau}g, V
  }_X \, \mathrm{d}t =
  \int_0^T \dual{g, V}_X \, \mathrm{d}t,
  \label{eq:Stau} \\
  \int_0^T \dual{
    (\D_{T-}^\alpha -\mathcal A^*) \mathcal S_{\alpha,\tau}^*g, V
  }_X \, \mathrm{d}t =
  \int_0^T \dual{g, V}_X \, \mathrm{d}t,
  \label{eq:S*tau}
\end{align}
for all $ V \in W_\tau(X) $. For any $ g \in L^1(0,T;X^*) $, define $ \mathcal
S_{1,\tau} g \in W_\tau(Y) $ and $ \mathcal S_{1,\tau}^* g \in W_\tau(Y) $,
respectively, by that
\begin{align}
  & \big(
    (\mathcal S_{1,\tau}g)(0+),
    V(0+)
  \big)_Y + \sum_{j=1}^{J-1} \big(
    (\mathcal S_{1,\tau} g)(t_j+) - (\mathcal S_{1,\tau} g)(t_j-),
    V(t_j+)
  \big)_Y \notag \label{eq:Stau-1} \\
  & \qquad {} - \int_0^T \dual{
    \mathcal A \mathcal S_{1,\tau} g, V
  }_X \, \mathrm{d}t =
  \int_0^T \dual{g, V}_X \, \mathrm{d}t, \\
  & \big(
    (\mathcal S_{1,\tau}^*g)(T-),
    V(T-)
  \big)_Y + \sum_{j=1}^{J-1} \big(
    (\mathcal S_{1,\tau}^* g)(t_j-) - (\mathcal S_{1,\tau}^* g)(t_j+),
    V(t_j-)
  \big)_Y \notag \\
  & \qquad {} - \int_0^T \dual{
    \mathcal A^* \mathcal S_{1,\tau}^* g, V
  }_X \, \mathrm{d}t =
  \int_0^T \dual{g, V}_X \, \mathrm{d}t,
  \label{eq:S*tau-1}
\end{align}
for all $ V \in W_\tau(X) $. We will present some properties of $ \mathcal
S_{\alpha,\tau} $, $ 0 < \alpha \leqslant 1 $, in \cref{ssec:Stau}.
\begin{remark}
  Scheme \cref{eq:Stau-1} is a famous discontinuous Galerkin method for
  parabolic equations (cf.~\cite{Eriksson1985}), and this scheme is a variant of
  the backward Euler difference scheme.
\end{remark}


\begin{remark} 
  We note that the idea of using the Galerkin methods to discretize the time
  fractional calculus operators was firstly developed by McLean and Mustapha
  \cite{Mclean2009Convergence,Mustapha2009Discontinuous,Mustapha2011Piecewise,Mustapha2014A}.
  The L1 scheme \cite{Lin2007,Sun2006} is widely used for the discretizations of
  the fractional diffusion equations. Jin et al.~\cite[Remark
  3]{Jin-maximal2018} discovered that the L1 scheme is equivalent to
  discretization \cref{eq:Stau} with uniform temporal grids. For the numerical
  analysis of discretization \cref{eq:Stau} with nonuniform temporal grids, we
  refer the reader to \cite{Li2019SIAM,Li-Wang-Xie2020}.
\end{remark}

Using the variational discretization concept proposed in \cite{Hinze2005}, we
consider the following temporally discrete problem:
\begin{equation}
  \label{eq:numer_opti}
  \min\limits_{U \in U_{\text{ad}}}
  J_{\alpha,\tau}(U) := \frac12 \nm{
    (\mathcal S_{\alpha,\tau}\mathcal R_{\theta_0} U)(T-) - y_d
  }_Y^2 + \frac\nu2 \nm{U}_{L^2(0,T;Z)}^2.
\end{equation}
Note that \cref{eq:zq} implies that $ (\mathcal S_{\alpha,\tau}\mathcal
R_{\theta_0}\cdot)(T-) $ is a bounded linear operator from $ L^2(0,T;Z) $ to $ Y
$. In addition, $ U_\text{ad} $ is a convex, bounded and closed subset of $
L^2(0,T;Z) $. Hence, applying \cite[Theorems 2.14 and 2.21]{Troltzsh2010} to
problem \cref{eq:numer_opti} yields the following theorem, by
\cref{lem:Stau-dual}.

\begin{theorem}
  \label{thm:regu-U}
  Problem \cref{eq:numer_opti} admits a unique solution $ U \in U_\text{ad} $,
  and the following first-order optimality condition holds:
  \begin{subequations}
  \begin{numcases}{}
    Y = \mathcal S_{\alpha,\tau}\mathcal R_{\theta_0} U, \label{eq:optim-Y} \\
    P = \mathcal S_{\alpha,\tau}^*\big( (Y(T-)-y_d)\widehat\delta_T \big),
    \label{eq:optim-P} \\
    \int_0^T \big(
      \mathcal R_{\theta_0}^* P(t) + \nu U(t), V(t)-U(t)
    \big)_Z \, \mathrm{d}t \geqslant 0
    \quad \text{ for all } V \in U_\text{ad},
    \label{eq:optim-U}
  \end{numcases}
  \end{subequations}
  where
  \begin{equation}
    \label{eq:delta_T-def}
    \widehat\delta_T :=
    \begin{cases}
      0 & \text{ if } 0 < t < T-\tau, \\ \tau^{-1} & \text{ if }
      T-\tau < t < T.
    \end{cases}
  \end{equation}
\end{theorem}

A simple modification of the proof of \cite[Theorem~4.3]{Li-Xie-Yan2020} yields
the following error estimate, by \cref{lem:conv-Stau}.
\begin{theorem} 
  \label{thm:conv}
  Assume that $ 0 < \alpha \leqslant 1 $. Let $ u $ and $ y $ be defined in
  \cref{thm:basic-regu}, and let $ U $ and $ Y $ be defined in
  \cref{thm:regu-U}. Then
  \begin{equation}
    \label{eq:conv}
    \begin{aligned}
      & \nm{(y-Y)(T-)}_Y + \sqrt\nu \nm{u-U}_{L^2(0,T;Z)} \\
      \leqslant{} & C_{c_0,\omega_0,\mathcal M_0,T}
      \left(
        \nm{y_d}_Y + \nm{\mathcal R_{\theta_0}}_{\mathcal L(Z,[X^*,Y]_{\theta_0})}
        \nm{u}_{L^\infty(0,T;Z)}
      \right) \times {} \\
      & \qquad \left(
        1/(\theta_0\alpha) +
        \sqrt{\varepsilon(\alpha,\theta_0,J)} +
        \varepsilon(\alpha,\theta_0,J) \tau^{\theta_0\alpha/2}
      \right) \tau^{\theta_0\alpha/2}.
    \end{aligned}
  \end{equation}
  where
  \begin{equation}
    \label{eq:vareps}
    \varepsilon(\alpha,\theta,J) :=
    \begin{cases}
      \frac1{\theta\alpha} +
      \frac{1-J^{\theta\alpha-1}}{1-\theta\alpha}
      & \text{ if } \theta\alpha \neq 1, \\
      \ln J &
      \text{ if } \theta\alpha = 1.
    \end{cases}
  \end{equation}
\end{theorem}

\subsection{Properties of \texorpdfstring{$\mathcal S_{\alpha,\tau}$}{}}
\label{ssec:Stau}
Assume that $ 0 < \alpha \leqslant 1 $ and $ g \in L^1(0,T;X^*) $. Define $
\{W_j\}_{j=1}^J \subset Y $ as follows: for any $ 1 \leqslant k \leqslant J $,
\begin{equation}
  \label{eq:L1}
  b_1 W_k + \sum_{j=1}^{k-1} (b_{k-j+1}-2b_{k-j}+b_{k-j-1})
  W_j - \tau^\alpha \mathcal A W_k =
  \tau^{\alpha-1} \int_{t_{k-1}}^{t_k} g(t) \, \mathrm{d}t
\end{equation}
in $ X^* $, where $ b_j := j^{1-\alpha}/\Gamma(2-\alpha)$ for each $ 0 \leqslant
j \leqslant J $. A straightforward computation yields that (cf.~\cite[Remark
3]{Jin-maximal2018})
\[
  (\mathcal S_{\alpha,\tau} g)(t_j-) = W_{j}
  \quad \forall 1 \leqslant j \leqslant J.
\]
Hence, we conclude from \cref{eq:L1,lem:z-A} that, for any $ 0 \leqslant \beta
\leqslant 1 $,
\begin{equation}
  \label{eq:zq}
  \mathcal S_{\alpha,\tau} \in \mathcal L\big(
    L^1(0,T;[X^*,Y]_{1-\beta}),\,
    L^\infty(0,T;[X,Y]_\beta)
  \big)
\end{equation}
and
\begin{equation}
  \label{eq:Stau-to-1}
  \lim_{\alpha \to {1-}} \nm{
    \mathcal S_{\alpha,\tau} - \mathcal S_{1,\tau}
  }_{
    \mathcal L(
    L^1(0,T;[X^*,Y]_{1-\beta}), \,
    L^\infty(0,T;[X,Y]_\beta)
    )
  } = 0.
\end{equation}
Symmetrically, for any $ 0 \leqslant \beta \leqslant 1 $ we have that
\begin{equation}
  \label{eq:zq*}
  \mathcal S_{\alpha,\tau}^* \in \mathcal L\big(
    L^1(0,T;[X^*,Y]_{1-\beta}),\,
    L^\infty(0,T;[X,Y]_\beta)
  \big)
\end{equation}
and
\begin{equation}
  \label{eq:Stau*-to-1}
  \lim_{\alpha \to {1-}} \nm{
    \mathcal S_{\alpha,\tau}^* - \mathcal S_{1,\tau}^*
  }_{
    \mathcal L(
    L^1(0,T;[X^*,Y]_{1-\beta}),\,
    L^\infty(0,T;[X,Y]_\beta)
    )
  } = 0.
\end{equation}

\begin{lemma}
  \label{lem:Stau-dual}
  Assume that $ 0 < \alpha \leqslant 1 $. For any $ f \in L^1(0,T;X^*) $ and $ g
  \in L^1(0,T;Y) $,
  \begin{equation}
    \label{eq:Stau-dual}
    \int_0^T (\mathcal S_{\alpha,\tau}f, \, g)_Y \, \mathrm{d}t
    = \int_0^T \dual{
      f,\, \mathcal S_{\alpha,\tau}^*g
    }_{X} \, \mathrm{d}t.
  \end{equation}
\end{lemma}
\begin{proof} 
  Assume that $ 0 < \alpha < 1 $. By \cref{eq:zq*} we have $ \mathcal
  S_{\alpha,\tau}^* g \in W_\tau(X) $, and then \cref{eq:S*tau} and the density
  of $ X $ in $ Y $ yield that
  \begin{align}
    \int_0^T \big(
      (\D_{T-}^\alpha - \mathcal A^*) \mathcal S_{\alpha,\tau}^*g, V
    \big)_Y \, \mathrm{d}t =
    \int_0^T (g, V)_Y \, \mathrm{d}t
  \end{align}
  for all $ V \in W_\tau(Y) $. Hence,
  \begin{align*}
    \int_0^T (\mathcal S_{\alpha,\tau}f, \, g)_Y \, \mathrm{d}t
    & = \int_0^T \big(
      \mathcal S_{\alpha,\tau}f, \,
      (\D_{T-}^\alpha - \mathcal A^*)
      \mathcal S_{\alpha,\tau}^* g
    \big)_Y \, \mathrm{d}t \\
    & = \int_0^T \dual{
      (\D_{0+}^\alpha - \mathcal A) \mathcal S_{\alpha,\tau}f, \,
      \mathcal S_{\alpha,\tau}^*g
    }_X \, \mathrm{d}t
    \quad\text{(by \cref{eq:dual,eq:A-ext})} \\
    & = \int_0^T \dual{
      f,\, \mathcal S_{\alpha,\tau}^*g
    }_{X} \, \mathrm{d}t
    \quad\text{(by \cref{eq:Stau}).}
  \end{align*}
  This proves \cref{eq:Stau-dual} for $ 0 < \alpha < 1 $. For the proof of
  \cref{eq:Stau-dual} with $ \alpha=1 $, we refer the reader to \cite[Chapter
  12]{Thomee2006}.
\end{proof}

\begin{lemma} 
  \label{lem:conv-Stau}
  Assume that $ 0 < \alpha, \theta \leqslant 1 $ and $ p \in \{1,\infty\} $. For
  any $ g \in L^p(0,T;[X^*,Y]_\theta) $ we have
  \begin{equation} 
    \label{eq:S-Stau-g}
    \nm{
      (\mathcal S_\alpha - \mathcal S_{\alpha,\tau}) g
    }_{L^p(0,T;Y)}
    \leqslant C_{c_0,\omega_0,\mathcal M_0}
    \varepsilon(\alpha,\theta,J) \tau^{\theta\alpha}
    \nm{g}_{L^p(0,T;[X^*,Y]_\theta)},
  \end{equation}
  and for any $ v \in Y $ we have
  \begin{small}
  \begin{equation}
    \label{eq:S-Stau-vdelta-l1}
    \nm{
      \mathcal S_\alpha(v\delta_0) -
      \mathcal S_{\alpha,\tau}(v\widehat\delta_0)
    }_{L^1(0,T;[X,Y]_\theta)}
    \leqslant C_{c_0,\omega_0,\mathcal M_0}
    \varepsilon(\alpha,\theta,J) \tau^{\theta\alpha}
    \nm{v}_Y,
  \end{equation}
  \end{small}
  where $ \varepsilon(\cdot,\cdot,\cdot) $ is defined by \cref{eq:vareps} and
  \begin{equation}
    \label{eq:delta_0}
    \widehat\delta_0(t) := \begin{cases}
      t_1^{-1} & \text{ if } 0 < t < t_1, \\ 0 & \text{ if } t_1 < t < T.
    \end{cases}
  \end{equation}
\end{lemma}

The main task of the rest of this subsection is to prove \cref{lem:conv-Stau}.
Firstly, we summarize some auxiliary results in \cite{Li-Xie-Yan2020}. Assume
that $ 0 < \alpha < 1 $. For any $ z \in \mathbb C \setminus \{0\} $ with $ -\pi
\leqslant \Im z \leqslant \pi $, define
\begin{equation}
  \psi_\alpha(z) :=
  (e^z-1)^2 \sum_{k=-\infty}^\infty
  (z + 2k\pi i)^{\alpha-2}.
\end{equation}
There exists $ \pi/2 < \omega^* \leqslant
\omega_0 $, depending only on $ \omega_0 $, such that
\begin{equation}
  \label{eq:0}
  e^{-z} \psi_\alpha(z) \in \Sigma_{\omega_0}
  \text{ for all $ z \in \Sigma_{\omega^*} $ with $ -\pi
  \leqslant \Im z \leqslant \pi $ }
\end{equation}
and that, for any $ z \in \Upsilon_{\omega^*} \setminus \{0\} $,
\begin{align}
  \snm{e^{-z}\psi_\alpha(z)} & \geqslant
  C_{\omega_0} \snm{z}^\alpha, \label{eq:psi>} \\
  \snm{\psi_\alpha(z) - z^\alpha} & \leqslant
  C_{\omega_0} \snm{z}^{\alpha+1}.  \label{eq:129}
\end{align}
Define
\begin{equation} 
  \label{eq:calE-def}
  \mathcal E_\alpha(t) := \tau^{-1}
  \mathcal E_{\alpha,\lfloor t/\tau \rfloor}, \quad t > 0,
\end{equation}
where $ \lfloor \cdot \rfloor $ is the floor function and
\begin{equation}
  \label{eq:calEj} \mathcal E_{\alpha,j} := \frac1{2\pi i}
  \int_{\Upsilon_{\omega^*}} e^{jz}
  R(\tau^{-\alpha}e^{-z}\psi_\alpha(z), \mathcal A)
  \, \mathrm{d}z, \quad j \in \mathbb N.
\end{equation}
Following the proof of \cite[Lemma 3.5]{Li-Xie-Yan2020}, we obtain that, for any
$ g \in L^1(0,T;Y) $,
\begin{equation} 
  \label{eq:Stau-g}
  (\mathcal S_{\alpha,\tau} g)(t_j-) =
  \int_0^{t_j} \mathcal E_\alpha(t_j-t) g(t) \, \mathrm{d}t
  \quad \forall 1 \leqslant j \leqslant J.
\end{equation}
Since \cref{eq:A-X*-Y-Y,eq:0,eq:psi>,eq:calE-def,eq:calEj} imply $ \nm{\mathcal
E_\alpha}_{L^\infty(0,T;\mathcal L(X^*,Y))} < \infty $, from \cref{eq:zq} and
the fact that $ L^1(0,T;Y) $ is dense in $ L^1(0,T;X^*) $ we conclude that
\cref{eq:Stau-g} holds for all $ g \in L^1(0,T;X^*) $.

Secondly, we present some auxiliary estimates in the following three lemmas.
\begin{lemma}
  \label{eq:A-X-Y}
  For any $ 0 < \alpha < 1 $, $ 0 \leqslant \theta \leqslant 1 $ and $ z \in
  \Upsilon_{\omega^*} \setminus \{0\} $,
  \begin{small}
  \begin{align} 
    \nm{
      e^{z} R(\tau^{-\alpha}z^\alpha,\mathcal A)
      - R(\tau^{-\alpha}e^{-z}\psi_\alpha(z),\mathcal A)
    }_{\mathcal L(Y,[X,Y]_\theta)} &
    \leqslant \frac{
      C_{c_0,\omega_0,\mathcal M_0} \snm{z}
    }{
      1+(\tau^{-\alpha}\snm{z}^\alpha)^\theta
    }, \label{eq:A-X-Y-diff} \\
    \nm{
      e^{z} R(\tau^{-\alpha}z^\alpha,\mathcal A) -
      R(\tau^{-\alpha}e^{-z}\psi_\alpha(z),\mathcal A)
    }_{\mathcal L([X^*,Y]_\theta,Y)} &
    \leqslant \frac{C_{c_0,\omega_0,\mathcal M_0}
    \snm{z}}{1+(\tau^{-\alpha}\snm{z}^\alpha)^\theta}.
    \label{eq:A-X*-Y-diff}
  \end{align}
  \end{small}
\end{lemma}
\begin{proof} 
  A straightforward computation gives
  \begin{align*}
    & e^{z}R(\tau^{-\alpha}z^\alpha,\mathcal A) -
    R(\tau^{-\alpha}e^{-z}\psi_\alpha(z),\mathcal A) \\
    ={} & \big( \tau^{-\alpha} \big( \psi_\alpha(z) - z^\alpha) +
    (1-e^{z})\mathcal A \big)
    R(\tau^{-\alpha} z^\alpha, \mathcal A)
    R(\tau^{-\alpha}e^{-z}\psi_\alpha(z),\mathcal A) \\
    ={} & \mathbb I_1 + \mathbb I_2,
  \end{align*}
  where
  \begin{align*}
    \mathbb I_1 &:=
    \tau^{-\alpha}(\psi_\alpha(z)-z^\alpha) R(\tau^{-\alpha}z^\alpha,\mathcal A)
    R(\tau^{-\alpha} e^{-z}\psi_\alpha(z),\mathcal A), \\
    \mathbb I_2 &:= (1-e^z)\mathcal A
    R(\tau^{-\alpha}z^\alpha,\mathcal A)
    R(\tau^{-\alpha} e^{-z}\psi_\alpha(z),\mathcal A).
  \end{align*}
  We conclude from \cref{eq:A-Y-X-Y,eq:0,eq:psi>} that, for any $ 0 \leqslant
  \beta \leqslant 1 $,
  \begin{align} 
    \nm{R(\tau^{-\alpha}z^\alpha, \mathcal A)}_{\mathcal L(Y,[X,Y]_\beta)} &
    \leqslant C_{\mathcal M_0} (1+(\tau^{-\alpha}\snm{z}^\alpha)^\beta)^{-1},
    \label{eq:lxy-1} \\
    \nm{
      R(\tau^{-\alpha}e^{-z}\psi_\alpha(z),\mathcal A)
    }_{\mathcal L(Y,[X,Y]_\beta)} &
    \leqslant C_{c_0,\omega_0,\mathcal M_0}
    (1+(\tau^{-\alpha}\snm{z}^\alpha)^\beta)^{-1}.
    \label{eq:lxy-2}
  \end{align}
  For $ \mathbb I_1 $ we have, by \cref{eq:129,eq:lxy-1,eq:lxy-2},
  \begin{small}
  \begin{align*}
    & \nm{\mathbb I_1}_{\mathcal L(Y,[X,Y]_\theta)} \\
    \leqslant{} &
    C_{\omega_0} \tau^{-\alpha} \snm{z}^{\alpha+1}
    \nm{
      R(\tau^{-\alpha}z^\alpha,\mathcal A)
    }_{\mathcal L(Y,[X,Y]_\theta)} \nm{
      R(\tau^{-\alpha}e^{-z}\psi_\alpha(z),\mathcal A)
    }_{\mathcal L(Y)} \\
    \leqslant{}& \frac{ C_{c_0,\omega_0,\mathcal M_0}
    \tau^{-\alpha} \snm{z}^{\alpha+1} }{
      \big( 1+\tau^{-\alpha}\snm{z}^\alpha \big)
      \big( 1+(\tau^{-\alpha}\snm{z}^\alpha)^\theta \big)
    } \leqslant
    \frac{
      C_{c_0,\omega_0,\mathcal M_0} \snm{z}
    }{
      1+(\tau^{-\alpha}\snm{z}^\alpha)^\theta
    }.
  \end{align*}
  \end{small}
  Since
  \begin{align*} 
    & \nm{
      \mathcal AR(\tau^{-\alpha}z^\alpha,\mathcal A)
      R(\tau^{-\alpha}e^{-z}\psi_\alpha(z), \mathcal A)
    }_{\mathcal L(Y,[X,Y]_\theta)} \\
    ={} & \nm{
      (\tau^{-\alpha}z^\alpha
      R(\tau^{-\alpha}z^\alpha,\mathcal A) - I) R(\tau^{-\alpha}e^{-z}\psi_\alpha(z),
      \mathcal A)
    }_{\mathcal L(Y,[X,Y]_\theta)} \\
    \leqslant{} &
    \nm{
      \tau^{-\alpha} z^\alpha
      R(\tau^{-\alpha}z^\alpha,\mathcal A)
    }_{\mathcal L(Y,[X,Y]_\theta)} \nm{
      R(\tau^{-\alpha}e^{-z}\psi_\alpha(z),\mathcal A)
    }_{\mathcal L(Y)} \\
    & \qquad {} + \nm{
      R(\tau^{-\alpha}e^{-z}\psi_\alpha(z),\mathcal A)
    }_{\mathcal (Y,[X,Y]_\theta)} \\
    \leqslant{} &
    \frac{C_{c_0,\omega_0,\mathcal M_0}}{
      1+(\tau^{-\alpha}\snm{z}^\alpha)^\theta
    } \quad\text{(by \cref{eq:lxy-1,eq:lxy-2}),}
  \end{align*}
  we obtain
  \[
    \nm{\mathbb I_2}_{\mathcal L(Y,[X,Y]_\theta)}
    \leqslant \frac{C_{c_0,\omega_0,\mathcal M_0} \snm{z}}{
      1+(\tau^{-\alpha}\snm{z}^\alpha)^\theta
    }.
  \]
  Combining the above estimates of $ \mathbb I_1 $ and $ \mathbb I_2 $ proves
  \cref{eq:A-X-Y-diff}. Since \cref{eq:A-X*-Y-diff} can be derived analogously,
  this completes the proof.
\end{proof}

\begin{lemma}
  \label{lem:E-calE}
  Assume that $ 0 < \alpha < 1 $ and $ 0 \leqslant \theta \leqslant 1 $. Then
  \begin{align}
    & \max_{1 \leqslant j \leqslant J}
    j^{2-\theta\alpha} \nm{E_\alpha(t_j) - \mathcal E_\alpha(t_j-)}_{
      \mathcal L(Y,[X,Y]_\theta)
    } \leqslant C_{c_0,\omega_0,\mathcal M_0}
    \tau^{\theta\alpha-1},
    \label{eq:E-calE-X*-Y-X-Y}\\
    & \max_{1 \leqslant j \leqslant J}
    j^{2-\theta\alpha} \nm{E_\alpha(t_j) - \mathcal E_\alpha(t_j-)}_{
      \mathcal L([X^*,Y]_\theta,Y)
    } \leqslant C_{c_0,\omega_0,\mathcal M_0}
    \tau^{\theta\alpha-1}.
    \label{eq:E-calE-X*-Y}
  \end{align}
\end{lemma}
\begin{proof} 
  For each $ 1 \leqslant j \leqslant J $, inserting $ t = t_j $ into
  \cref{eq:E-def} yields
  \[ 
    E_\alpha(t_j) = \frac1{2\pi i}
    \int_{\Gamma_{\omega^*}} e^{t_jz} R(z^\alpha,\mathcal A) \, \mathrm{d}z =
    \frac{\tau^{-1}}{2\pi i} \int_{\Gamma_{\omega^*}} e^{jz}
    R(\tau^{-\alpha}z^\alpha, \mathcal A) \, \mathrm{d}z,
  \]
  and so we conclude from \cref{eq:calE-def,eq:calEj} that
  \[
    E_\alpha(t_j) - \mathcal E_\alpha(t_j-) = \mathbb I_1 + \mathbb I_2,
  \]
  where
  \begin{align*}
    \mathbb I_1 &:= \frac{\tau^{-1}}{2\pi i}
    \int_{\Gamma_{\omega^*}\setminus\Upsilon_{\omega^*}} e^{jz}
    R(\tau^{-\alpha}z^\alpha,\mathcal A) \, \mathrm{d}z, \\ \mathbb I_2 &:=
    \frac{\tau^{-1}}{2\pi i} \int_{\Upsilon_{\omega^*}} e^{(j-1)z} \big( e^z
    R(\tau^{-\alpha}z^\alpha,\mathcal A) -
    R(\tau^{-\alpha}e^{-z}\psi_\alpha(z),\mathcal A) \big) \, \mathrm{d}z.
  \end{align*}
  A straightforward computation gives
  \begin{align*} 
    & \nm{\mathbb I_1}_{
      \mathcal L(Y,[X,Y]_\theta)
    } \\
    \leqslant{} & C_{c_0,\mathcal M_0} \tau^{-1}
    \int_{\pi/\sin\omega^*}^\infty e^{j\cos\omega^* r}
    \big( 1+(\tau^{-\alpha}r^\alpha)^{\theta} \big)^{-1} \,
    \mathrm{d}r \quad\text{(by \cref{eq:A-Y-X-Y})} \\
    <{} & C_{c_0,\mathcal M_0} \tau^{\theta\alpha-1}
    \int_{\pi/\sin\omega^*}^\infty e^{j\cos\omega^* r}
    r^{-\theta\alpha} \, \mathrm{d}r \\
    <{} & C_{c_0,\mathcal M_0} \tau^{\theta\alpha-1}
    \int_{\pi/\sin\omega^*}^\infty e^{j\cos\omega^* r} \, \mathrm{d}r \\
    <{} & C_{c_0,\omega_0,\mathcal M_0} \tau^{\theta\alpha-1}
    j^{-1} e^{j\pi\cot\omega^*}
  \end{align*} and
  \begin{align*} 
    & \nm{\mathbb I_2}_{
      \mathcal L(Y,[X,Y]_\theta)
    } \\ <{} &
    C_{c_0,\omega_0,\mathcal M_0} \tau^{-1}
    \int_0^{\pi/\sin\omega^*} e^{(j-1)\cos\omega^* r}
    r(1 + (\tau^{-\alpha} r^\alpha)^{\theta})^{-1} \, \mathrm{d}r
    \quad\text{(by \cref{eq:A-X-Y-diff})} \\
    <{} & C_{c_0,\omega_0,\mathcal M_0}
    \tau^{\theta\alpha-1} \int_0^{\pi/\sin\omega^*}
    e^{(j-1)\cos\omega^* r} r^{1-\theta\alpha} \, \mathrm{d}r \\
    <{} & C_{c_0,\omega_0,\mathcal M_0}
    \tau^{\theta\alpha-1} j^{\theta\alpha-2}.
  \end{align*}
  Together the above estimates of $ \mathbb I_1 $ and $ \mathbb I_2 $ yields
  \cref{eq:E-calE-X*-Y-X-Y}. Since \cref{eq:E-calE-X*-Y} can be proved
  analogously by \cref{eq:A-X*-Y-Y,eq:A-X*-Y-diff}, this completes the proof.
\end{proof}

\begin{lemma}
  \label{lem:E-calE-l1}
  For any $ 0 < \alpha < 1 $ and $ 0 < \theta \leqslant 1 $,
  \begin{align}
    \nm{E_\alpha-\mathcal E_{\alpha}}_{
      L^1(0,T;\mathcal L(Y,[X,Y]_\theta))
    } &
    \leqslant C_{c_0,\omega_0,\mathcal M_0}
    \Big(
      \frac1{\theta\alpha} +
      \frac{1-J^{\theta\alpha-1}}{1-\theta\alpha}
    \Big) \tau^{\theta\alpha} , \label{eq:int-E-calE-high} \\
    \nm{E_\alpha-\mathcal E_{\alpha}}_{
      L^1(0,T;\mathcal L([X^*,Y]_\theta,Y))
    } & \leqslant
    C_{c_0,\omega_0,\mathcal M_0} \Big(
      \frac1{\theta\alpha} +
      \frac{1-J^{\theta\alpha-1}}{1-\theta\alpha}
    \Big) \tau^{\theta\alpha}. \label{eq:int-E-calE}
  \end{align}
\end{lemma}
\begin{proof} 
  By \cref{eq:E-X-Y} we have
  \begin{equation} 
    \label{eq:E-Et1} \nm{E_\alpha-E_\alpha(t_1)}_{
      L^1(0,t_1;\mathcal L(Y,[X,Y]_\theta))
    } < C_{c_0,\omega_0,\mathcal M_0}
    \tau^{\theta\alpha} / (\theta \alpha),
  \end{equation}
  and a straightforward calculation gives
  \begin{small}
  \begin{align}
    & \sum_{j=2}^J \nm{ E_\alpha - E_\alpha(t_j) }_{
      L^1(t_{j-1},t_j;\mathcal L(Y,[X,Y]_\theta))
    } \notag \\
    \leqslant{} & \tau
    \nm{E_\alpha'}_{ L^1(t_1,T;\mathcal L(Y,[X,Y]_\theta)) } \leqslant
    C_{c_0,\omega_0,\mathcal M_0} \tau \int_{t_1}^T t^{\theta\alpha-2} \, \mathrm{d}t
    \quad \text{(by \cref{eq:E'-X-Y})} \notag \\
    ={} & C_{c_0,\omega_0,\mathcal M_0}
    \tau^{\theta\alpha} (1-J^{\theta\alpha-1}) / (1-\theta\alpha).
  \end{align}
  \end{small}
  It follows that
  \begin{small}
  \begin{align*}
    & \sum_{j=1}^J \nm{E_\alpha-E_\alpha(t_j)}_{
      L^1(t_{j-1},t_j;\mathcal L(Y,[X,Y]_\theta))
    } \\
    \leqslant{}& C_{c_0,\omega_0,\mathcal M_0} \Big(
      \frac1{\theta\alpha} +
      \frac{1-J^{\theta\alpha-1}}{1-\theta\alpha}
    \Big) \tau^{\theta\alpha}.
  \end{align*}
  \end{small}
  In addition, by \cref{eq:E-calE-X*-Y-X-Y},
  \begin{align*}
    \sum_{j=1}^J \tau \nm{E_\alpha(t_j) - \mathcal E_\alpha(t_j-)}_{ \mathcal
    L(Y,[X,Y]_\theta) } & \leqslant C_{c_0,\omega_0,\mathcal M_0}
    \tau^{\theta\alpha} \sum_{j=1}^J j^{\theta\alpha-2} \\ & \leqslant
    C_{c_0,\omega_0,\mathcal M_0} \tau^{\theta\alpha} (1-J^{\theta\alpha-1}) /
    (1-\theta\alpha).
  \end{align*}
  Consequently,
  \begin{align*}
    & \nm{E_\alpha-\mathcal E_\alpha}_{
      L^1(0,T;\mathcal L(Y,[X,Y]_\theta))
    } \\
    \leqslant{} & \sum_{j=1}^J
    \Big( \nm{E_\alpha-E_\alpha(t_j)}_{
      L^1(t_{j-1},t_j;\mathcal L(Y,[X,Y]_\theta))
    } + {} \\ &
    \qquad\qquad\qquad \tau \nm{
      E_\alpha(t_j)-\mathcal E_\alpha(t_j-)
    }_{
      L^1(t_{j-1},t_j;\mathcal L(Y,[X,Y]_\theta))
    } \Big) \\
    \leqslant{} & C_{c_0,\omega_0,\mathcal M_0}
    \Big(
      \frac1{\theta\alpha} +
      \frac{1-J^{\theta\alpha-1}}{1-\theta\alpha}
    \Big) \tau^{\theta\alpha},
  \end{align*}
  which proves \cref{eq:int-E-calE-high}. Since \cref{eq:int-E-calE} can be
  proved analogously by \cref{eq:E-X*-Y,eq:E'-X*-Y,eq:E-calE-X*-Y}, this
  completes the proof.
\end{proof}

Thirdly, we prove that \cref{eq:S-Stau-g} holds for $ 0 < \alpha < 1 $ and $ p=1
$.
\begin{lemma}
  \label{lem:conv-l1}
  Assume that $ 0 < \alpha < 1 $ and $ 0 < \theta \leqslant 1 $. For any $ g \in
  L^1(0,T;[X^*,Y]_\theta) $, we have
  \begin{equation}
    \label{eq:conv-l1}
    \nm{
      (\mathcal S_\alpha - \mathcal S_{\alpha,\tau})g
    }_{L^1(0,T;Y)} \leqslant
    C_{c_0,\omega_0,\mathcal M_0} \Big(
      \frac1{\theta\alpha} +
      \frac{1-J^{\theta\alpha-1}}{1-\theta\alpha}
    \Big) \tau^{\theta\alpha}
    \nm{g}_{L^1(0,T;[X^*,Y]_\theta)}.
  \end{equation}
\end{lemma}
\begin{proof} 
  {\it Step 1.} Let us prove
  \begin{equation}
    \label{eq:300}
    \begin{aligned}
      & \int_0^T \Nm{
        \int_0^t (E_\alpha(t-s) - E_\alpha(t-s+\tau))g(s) \, \mathrm{d}s
      }_Y \, \mathrm{d}t \\
      \leqslant{} &
      C_{c_0,\omega_0,\mathcal M_0}
      \left(
        \frac1{\theta\alpha} +
        \frac{1-J^{\theta\alpha-1}}{1-\theta\alpha}
      \right) \tau^{\theta\alpha}
      \nm{g}_{L^1(0,T;[X^*,Y]_\theta)}.
    \end{aligned}
  \end{equation}
  A straightforward computation gives
  \[
    \int_0^\tau \nm{
      E_\alpha(t) - E_\alpha(t+\tau)
    }_{\mathcal L([X^*,Y]_\theta,Y)} \, \mathrm{d}t
    \leqslant
    C_{c_0,\omega_0,\mathcal M_0}
    \tau^{\theta\alpha} / (\theta\alpha) \quad \text{(by \cref{eq:E-X*-Y})}
  \]
  and
  \begin{align*}
    & \int_\tau^T \nm{
      E_\alpha(t) - E_\alpha(t+\tau)
    }_{\mathcal L([X^*,Y]_\theta,Y)} \, \mathrm{d}t \\
    \leqslant{} &
    C_{c_0,\omega_0,\mathcal M_0} \tau
    \int_\tau^T t^{\theta\alpha-2} \, \mathrm{d}t
    \quad\text{(by \cref{eq:E'-X*-Y})} \\
    \leqslant{} & C_{c_0,\omega_0,\mathcal M_0}
    \tau^{\theta\alpha} (1-J^{\theta\alpha-1}) / (1-\theta\alpha).
  \end{align*}
  It follows that
  \[
    \int_0^T \nm{
      E_\alpha(t) - E_\alpha(t+\tau)
    }_{\mathcal L([X^*,Y]_\theta,Y)}
    \leqslant
    C_{c_0,\omega_0,\mathcal M_0}
    \left(
      \frac1{\theta\alpha} +
      \frac{1-J^{\theta\alpha-1}}{1-\theta\alpha}
    \right) \tau^{\theta\alpha}.
  \]
  Hence, \cref{eq:300} follows from the estimate
  \begin{align*} 
    & \int_0^T \Big\|
    \int_0^t \big(
      E_\alpha(t-s) - E_\alpha(t-s+\tau)
    \big) g(s) \, \mathrm{d}s \Big\|_Y \, \mathrm{d}t \\
    \leqslant{} &
    \int_0^T
    \int_0^t \nm{
      E_\alpha(t-s) - E_\alpha(t-s+\tau)
    }_{\mathcal L([X^*,Y]_\theta,Y)} \nm{g(s)}_{[X^*,Y]_\theta} \, \mathrm{d}s
    \, \mathrm{d}t \\
    ={} &
    \int_0^T \int_s^T \nm{
      E_\alpha(t-s) - E_\alpha(t-s+\tau)
    }_{\mathcal L([X^*,Y]_\theta,Y)}
    \nm{g(s)}_{[X^*,Y]_\theta} \, \mathrm{d}t \, \mathrm{d}s \\
    ={} &
    \int_0^T \int_0^{T-s} \nm{
      E_\alpha(r) - E_\alpha(r+\tau)
    }_{\mathcal L([X^*,Y]_\theta,Y)} \, \mathrm{d}r \,
    \nm{g(s)}_{[X^*,Y]_\theta}  \, \mathrm{d}s \\
    \leqslant{} &
    \nm{g}_{L^1(0,T;[X^*,Y]_\theta)}
    \int_0^T \nm{
      E_\alpha(r) - E_\alpha(r+\tau)
    }_{\mathcal L([X^*,Y]_\theta,Y)} \, \mathrm{d}s.
  \end{align*}

  {\it Step 2.} Let us prove that
  \begin{small}
  \begin{equation}
    \label{eq:301}
    \int_0^T \Big\|
    \int_0^t E_\alpha(t-s+\tau) g(s) - G(t)
    \Big\|_Y \, \mathrm{d}t \leqslant
    C_{c_0,\omega_0,\mathcal M_0}
    \frac{1-J^{\theta\alpha-1}}{1-\theta\alpha}
    \tau^{\theta\alpha}
    \nm{g}_{L^1(0,T;[X^*,Y]_\theta)},
  \end{equation}
  \end{small}
  where $ G \in W_\tau(Y) $ is defined by
  \begin{equation}
    \label{eq:G}
    G(t_j-) := \sum_{k=1}^j E_\alpha(t_j-t_k+\tau)
    \int_{t_{k-1}}^{t_k} g(s) \, \mathrm{d}s,
    \quad 1 \leqslant j \leqslant J.
  \end{equation}
  For any $ t_{j-1} < t < t_j $ with $ 1 \leqslant j \leqslant J $, by
  \cref{eq:E-X*-Y} we have
  \begin{align*}
    & \Big\|
    \int_{t_{j-1}}^t E_\alpha(t-s+\tau) g(s) \, \mathrm{d}s -
    E_\alpha(\tau) \int_{t_{j-1}}^{t_j} g(s) \, \mathrm{d}s
    \Big\|_Y \\
    \leqslant{} & C_{c_0,\omega_0,\mathcal M_0}
    \tau^{\theta\alpha-1} \nm{g}_{L^1(t_{j-1},t_j;[X^*,Y]_\theta)},
  \end{align*}
  and by \cref{eq:E'-X*-Y} we have
  \begin{align*}
    & \Big\|
    \sum_{k=1}^{j-1} \int_{t_{k-1}}^{t_k} \big(
      E_\alpha(t-s+\tau) -E_\alpha(t_j-t_k+\tau)
    \big) g(s) \, \mathrm{d}s
    \Big\|_Y \\
    \leqslant{} & C_{c_0,\omega_0,\mathcal M_0}
    \sum_{k=1}^{j-1} \tau(t_j-t_k+\tau)^{\theta\alpha-2}
    \nm{g}_{L^1(t_{k-1},t_k;[X^*,Y]_\theta)} \\
    ={} & C_{c_0,\omega_0,\mathcal M_0}
    \tau^{\theta\alpha-1} \sum_{k=1}^{j-1}(j-k+1)^{\theta\alpha-2}
    \nm{g}_{L^1(t_{k-1},t_k;[X^*,Y]_\theta)}.
  \end{align*}
  Hence, for each $ t_{j-1} < t < t_j $ with $ 1 \leqslant j \leqslant J $,
  \begin{align*}
    & \Big\|
    \int_0^t E_\alpha(t-s+\tau) g(s) \, \mathrm{d}s -G(t_j-)
    \Big\|_{Y} \\
    \leqslant{} &
    \Big\|
    \int_{t_{j-1}}^t E_\alpha(t-s+\tau) g(s) \, \mathrm{d}s -
    E_\alpha(\tau) \int_{t_{j-1}}^{t_j} g(s) \, \mathrm{d}s
    \Big\|_Y + \\
    & \quad{} + \Nm{
      \sum_{k=1}^{j-1} \int_{t_{k-1}}^{t_k} \big(
        E_\alpha(t-s+\tau)-E_\alpha(t_j-t_k+\tau)
      \big) g(s) \, \mathrm{d}s
    }_Y \quad\text{(by \cref{eq:G})} \\
    \leqslant{} & C_{c_0,\omega_0,\mathcal M_0}
    \tau^{\theta\alpha-1}
    \sum_{k=1}^j (j-k+1)^{\theta\alpha-2}
    \nm{g}_{L^1(t_{k-1},t_k;[X^*,Y]_\theta)}.
  \end{align*}
  It follows that, for each $ 1 \leqslant j \leqslant J $,
  \begin{align*} 
    & \int_{t_{j-1}}^{t_j} \Big\|
    \int_0^t E_\alpha(t+\tau-s) g(s) \, \mathrm{d}s -
    G(t)
    \Big\|_Y \, \mathrm{d}t \\
    \leqslant{} & C_{c_0,\omega_0,\mathcal M_0}
    \tau^{\theta\alpha} \sum_{k=1}^j (j-k+1)^{\theta\alpha-2}
    \nm{g}_{L^1(t_{k-1},t_k;[X^*,Y]_\theta)}.
  \end{align*}
  Therefore,
  \begin{align*}
    & \int_0^T \Big\|
    \int_0^t E_\alpha(t+\tau-s) g(s) \, \mathrm{d}s - G(t)
    \Big\|_Y \, \mathrm{d}t \\
    \leqslant{} & C_{c_0,\omega_0,\mathcal M_0}
    \tau^{\theta\alpha} \sum_{j=1}^J
    \sum_{k=1}^j (j-k+1)^{\theta\alpha-2}
    \nm{g}_{L^1(t_{k-1},t_k;[X^*,Y]_\theta)} \\
    ={} & C_{c_0,\omega_0,\mathcal M_0}
    \tau^{\theta\alpha} \sum_{k=1}^J
    \sum_{j=k}^J (j-k+1)^{\theta\alpha-2}
    \nm{g}_{L^1(t_{k-1},t_k;[X^*,Y]_\theta)} \\
    \leqslant{} & C_{c_0,\omega_0,\mathcal M_0}
    \tau^{\theta\alpha} \nm{g}_{L^1(0,T;[X^*,Y]_\theta)}
    \sum_{m=1}^J m^{\theta\alpha-2}.
  \end{align*}
  The desired estimate \cref{eq:301} then follows from the simple inequality
  \begin{equation}
    \label{eq:2000}
    \sum_{m=1}^J m^{\theta\alpha-2} <
    1 + \frac{1-J^{\theta\alpha-1}}{1-\theta\alpha}.
  \end{equation}

  {\it Step 3.} Let us prove that
  \begin{equation}
    \label{eq:302}
    \int_0^T \nm{
      (\mathcal S_{\alpha,\tau} g)(t) - G(t)
    }_Y \, \mathrm{d}t \leqslant
    C_{c_0,\omega_0,\mathcal M_0}
    \frac{1-J^{\theta\alpha-1}}{1-\theta\alpha}
    \tau^{\theta\alpha}
    \nm{g}_{L^1(0,T;[X^*,Y]_\theta)}.
  \end{equation}
  Noting that $ \mathcal S_{\alpha,\tau}g $ and $ G $ are piecewise constant, we
  have
  \begin{small}
  \begin{align*}
    & \int_0^T \Nm{
      (\mathcal S_{\alpha,\tau}g)(t) - G(t)
    }_Y \, \mathrm{d}t \\
    ={} &
    \sum_{j=1}^J \tau \Nm{
      (\mathcal S_{\alpha,\tau}g)(t_j-) - G(t_j-)
    }_Y \\
    \leqslant{} &
    \sum_{j=1}^J \tau \Big\|
      \sum_{k=1}^j \big(
        \mathcal E_\alpha((t_j\!-\!t_k\!+\!\tau)-) -
        E_\alpha(t_j\!-\!t_k\!+\!\tau)
      \big) \int_{t_{k-1}}^{t_k} g(s) \, \mathrm{d}s
    \Big\|_Y \,\text{(by \cref{eq:Stau-g,eq:G})} \\
    \leqslant{} &
    \sum_{j=1}^J \tau \sum_{k=1}^j
    \nm{\mathcal E_\alpha((t_j-t_k+\tau)-) -
    E_\alpha(t_j-t_k+\tau)}_{
      \mathcal L([X^*,Y]_\theta,Y)
    } \int_{t_{k-1}}^{t_k} \nm{g(s)}_{[X^*,Y]_\theta} \, \mathrm{d}s \\
    = & \tau
    \sum_{k=1}^J \sum_{j=k}^J
    \nm{\mathcal E_\alpha((t_j-t_k+\tau)-) -
    E_\alpha(t_j-t_k+\tau)}_{
      \mathcal L([X^*,Y]_\theta,Y)
    } \int_{t_{k-1}}^{t_k} \nm{g(s)}_{[X^*,Y]_\theta} \, \mathrm{d}s \\
    \leqslant{} &
    \tau \nm{g}_{L^1(0,T;[X^*,Y]_\theta)}
    \sum_{m=1}^J \nm{
      \mathcal E_\alpha(t_m-) - E(t_m)
    }_{\mathcal L([X^*,Y]_\theta,Y)} \\
    \leqslant{} & C_{c_0,\omega_0,\mathcal M_0}
    \tau^{\theta\alpha} \nm{g}_{L^1(0,T;[X^*,Y]_\theta)}
    \sum_{m=1}^J m^{\theta\alpha-2} \quad\text{(by \cref{eq:E-calE-X*-Y})}.
  \end{align*}
  \end{small}
  Hence, \cref{eq:302} follows from \cref{eq:2000}.

  {\it Step 4.} By \cref{eq:Sg-l1} we have
  \begin{align*}
    & \nm{
      (\mathcal S_\alpha - \mathcal S_{\alpha,\tau})g
    }_{L^1(0,T;Y)} \\
    \leqslant{} &
    \int_0^T \Big\|
    \int_0^t (E_\alpha(t-s) - E_\alpha(t-s+\tau))g(s) \, \mathrm{d}s
    \Big\|_Y \, \mathrm{d}t \\
    & \quad {} +
    \int_0^T \Big\|
      \int_0^t E_\alpha(t-s+\tau) g(s) \, \mathrm{d}s - G(t)
    \|_Y \, \mathrm{d}t \\
    & \quad {} +
    \int_0^T \Nm{
      (\mathcal S_{\alpha,\tau}g)(t) - G(t)
    }_Y \, \mathrm{d}t.
  \end{align*}
  Therefore, combining \cref{eq:300,eq:301,eq:302} proves \cref{eq:conv-l1} and
  thus concludes the proof.
\end{proof}

Fourthly, let us prove that \cref{eq:S-Stau-g} holds for $ 0 < \alpha < 1 $ and
$ p=\infty $.
\begin{lemma}
  \label{lem:sj}
  Assume that $ 0 < \alpha < 1 $ and $ 0 < \theta \leqslant 1 $. For any $ g \in
  L^\infty(0,T;[X^*,Y]_\theta) $ we have
  \begin{equation}
    \label{eq:sj}
    \begin{aligned}
      & \nm{
        (\mathcal S_\alpha - \mathcal S_{\alpha,\tau})g
      }_{L^\infty(0,T;Y)} \\
      \leqslant{} &
      C_{c_0,\omega_0,\mathcal M_0}
      \left(
        \frac1{\theta\alpha} + \frac{1-J^{\theta\alpha-1}}{1-\theta\alpha}
      \right) \tau^{\theta\alpha}
      \nm{g}_{L^\infty(0,T;[X^*,Y]_\theta)}.
    \end{aligned}
  \end{equation}
\end{lemma}
\begin{proof} 
  By virtue of \cref{eq:Sg-l1,eq:Stau-g} we have that
  \begin{align*} 
    & \max_{1
    \leqslant j \leqslant J} \nm{
      (\mathcal S_\alpha g)(t_j)-(\mathcal S_{\alpha,\tau} g)(t_j-)
    }_Y \\
    \leqslant{} &
    \nm{E_\alpha-\mathcal E_\alpha}_{L^1(0,T;\mathcal L([X^*,Y]_\theta,Y))}
    \nm{g}_{L^\infty(0,T;[X^*,Y]_\theta)},
  \end{align*}
  so that \cref{eq:int-E-calE} implies that
  \begin{align*}
    & \max_{1 \leqslant j \leqslant J}
    \nm{
      (\mathcal S_\alpha g)(t_j)-(\mathcal S_{\alpha,\tau} g)(t_j-)
    }_Y \\
    \leqslant{} &
    C_{c_0,\omega_0,\mathcal M_0} \left(
      \frac1{\theta\alpha} +
      \frac{1-J^{\theta\alpha-1}}{1-\theta\alpha}
    \right) \nm{g}_{L^\infty(0,T;[X^*,Y]_\theta)}.
  \end{align*}
  It remains therefore to prove that
  \begin{align*}
    & \max_{1 \leqslant j \leqslant J}
    \nm{
      \mathcal S_\alpha g - (\mathcal S_\alpha g)(t_j)
    }_{L^\infty(t_{j-1},t_j;Y)} \\
    \leqslant{} &
    C_{c_0,\omega_0,\mathcal M_0}
    \left(
      \frac1{\theta\alpha} + \frac{1-J^{\theta\alpha-1}}{1-\theta\alpha}
    \right) \tau^{\theta\alpha}
    \nm{g}_{L^\infty(0,T;[X^*,Y]_\theta)}.
  \end{align*}
  But this is easily derived by \cref{eq:E-X*-Y,eq:E'-X*-Y,eq:Sg-l1}; see the
  proof of \cite[Thoerem 2.6]{Diethelm2010} for the relevant techniques. This
  completes the proof.
\end{proof}

Finally, we are in a position to conclude the proof of \cref{lem:conv-Stau} as
follows.

\medskip\noindent{\bf Proof of \cref{lem:conv-Stau}.} Let us first prove
  \cref{eq:S-Stau-g}. By \cref{lem:conv-l1,lem:sj} we have that
  \cref{eq:S-Stau-g} holds for all $ 0 < \alpha < 1 $ and $ p \in \{1,\infty\}
  $, so that passing to the limit $ \alpha \to {1} $ yields, by
  \cref{lem:S-to-1,eq:Stau-to-1}, that \cref{eq:S-Stau-g} holds for all $ 0 <
  \alpha \leqslant 1 $ and $ p \in \{1,\infty\} $.

  Then let us prove \cref{eq:S-Stau-vdelta-l1}. Assume that $ 0 < \alpha < 1 $.
  Combining \cref{eq:Stau-g,eq:delta_0} gives that
  \[
    (\mathcal S_{\alpha,\tau} (v\widehat\delta_0))(t_j-) =
    \mathcal E_\alpha(t_j-) v, \quad 1 \leqslant j
    \leqslant J,
  \]
  and so \cref{eq:Sdelta} implies
  \begin{align*}
    & \nm{
      \mathcal S_\alpha(v\delta_0) -
      \mathcal S_{\alpha,\tau}(v\widehat\delta_0)
    }_{L^1(0,T;[X,Y]_\theta)} \\
    \leqslant{} &
    \nm{E_\alpha-\mathcal E_\alpha}_{
      L^1(0,T;\mathcal L(Y,[X,Y]_\theta))
    } \nm{v}_{Y}.
  \end{align*}
  Therefore, \cref{eq:int-E-calE-high} proves that \cref{eq:S-Stau-vdelta-l1}
  holds for each $ 0 < \alpha < 1 $. A simple modification of the proof of
  \cref{eq:lbj} gives
  \[
    \lim_{\alpha \to {1-}} \nm{E_\alpha - E_1}_{
      L^1(0,T;\mathcal L(Y,[X,Y]_\theta))
    } = 0,
  \]
  so that \cref{eq:Sdelta} implies
  \[
    \lim_{\alpha \to {1-}}
    \nm{
      \mathcal S_\alpha(v\delta_0) -
      \mathcal S_1(v\delta_0)
    }_{L^1(0,T;[X,Y]_\theta)} = 0.
  \]
  Moreover, \cref{eq:Stau-to-1} yields
  \[
    \lim_{\alpha \to {1-}}
    \nm{
      \mathcal S_{\alpha,\tau}(v\widehat\delta_0) -
      \mathcal S_{1,\tau}(v\widehat\delta_0)
    }_{L^1(0,T;[X,Y]_\theta)} = 0.
  \]
  Therefore, passing to the limit $ \alpha \to {1-} $ in
  \cref{eq:S-Stau-vdelta-l1} yields that \cref{eq:S-Stau-vdelta-l1} holds with $
  \alpha=1 $. This completes the proof of \cref{lem:conv-Stau}.
\hfill\ensuremath{\blacksquare}

\section{A Dirichlet boundary control problem}
\label{sec:appl}
Assume that $ 0 < \alpha \leqslant 1 $, $ 0 < T < \infty $, and $ \Omega \subset
\mathbb R^d $ ($d=2,3$) is a bounded convex polygonal domain with boundary $
\partial\Omega $. Define
\[
  U_{\text{ad}} := \left\{
    v \in L^2(0,T;L^2(\partial\Omega)):\
    u_* \leqslant v(x,t) \leqslant u^*
    \text{ for a.e.~} (x,t) \in \partial\Omega \times (0,T)
  \right\},
\]
where $ u_* $ and $ u^* $ are two given constants. For any $ y \in
C((0,T];L^2(\Omega)) $ and $ u \in L^2(0,T;L^2(\partial\Omega)) $, define
\begin{equation}
  J_\alpha(y,u) := \frac12 \nm{y(T) - y_d}_{L^2(\Omega)}^2 +
  \frac\nu2 \nm{u}_{L^2(0,T;L^2(\partial\Omega))}^2,
\end{equation}
where $ y_d \in L^2(\Omega) $ and $ \nu > 0 $ is a regularization parameter. We
are concerned with the following optimal Dirichlet boundary control problem:
\begin{equation}
  \label{eq:model}
  \boxed{
    \begin{aligned}
      \text{Minimize } J_\alpha(y,u)
      \text{ subject to $ u \in U_\text{ad} $ and} \\
      \begin{cases}
        (\partial_{0+}^\alpha - \Delta) y = 0
        & \text{ in  } \Omega \times (0,T) \\
        \qquad\qquad\,\, y = u
        & \text{ on  } \partial\Omega \times (0,T) \\
        \qquad\,\, y(\cdot,0) = 0
        & \text{ in } \Omega.
      \end{cases}
    \end{aligned}
  }
\end{equation}
Here, $ \partial_{0+}^\alpha $, a fractional partial differential operator, is
the scalar-valued version of $ \D_{0+}^\alpha $ with respect to the time
variable $ t $.

To apply the theory in the previous section to problem \cref{eq:model}, we will
use the following settings:
\begin{align*}
  & \mathcal A := \Delta; \quad \mathcal A^* := \Delta; \quad
  X := H_0^1(\Omega) \cap H^2(\Omega); \quad Y := L^2(\Omega);
  \quad Z := L^2(\partial\Omega);
\end{align*}
the operator $ \mathcal R_{\theta_0}: Z \to [X^*,Y]_{\theta_0} $, $ 0 < \theta_0
< 1/4 $, is defined by that
\begin{equation} 
  \dual{\mathcal R_{\theta_0} w, v}_{[X,Y]_{\theta_0}} :=
  -\dual{w, \partial_{\bm n} v}_{\partial\Omega}
\end{equation}
for all $ w \in Z $ and $ v \in [X,Y]_{\theta_0} $, where $ \partial_{\bm n} v $
is the outward normal derivative of $ v $ on $ \partial\Omega $. By the
well-known trace inequality that
\begin{equation}
  \label{eq:trace}
  \nm{v}_{Z}
  \leqslant \frac{C_\Omega}{\sqrt{\epsilon}}
  \nm{v}_{[X, Y]_{3/4-\epsilon}}
  \quad \text{ for all }
  v \in [X,Y]_{3/4-\epsilon} \text{ with }
  0 < \epsilon \leqslant 3/4,
\end{equation}
we readily conclude that, for any $ 0 < \theta_0 < 1/4 $,
\begin{equation}
  \label{eq:R-esti}
  \nm{\mathcal R_{\theta_0}}_{
    \mathcal L(Z,[X^*,Y]_{\theta_0})
  } \leqslant \frac{C_\Omega}{\sqrt{1-4\theta_0}}.
\end{equation}

\begin{remark} 
  For the techniques to prove \cref{eq:trace}, we refer the reader to \cite[§~3,
  Ch.~VI]{Stein1970}, \cite[Lemmas 16.1 and 23.1]{Tartar2007} and
  \cite[Corollary 4.37]{Lunardi2018}.
\end{remark}

Let $ u $ and $ y $ be defined in \cref{thm:basic-regu}, and let $ U $ and $ Y $
be defined in \cref{thm:regu-U}. A straightforward calculation gives, by
\cref{eq:conv,eq:R-esti}, that
\begin{align*}
  & \nm{(y-Y)(T-)}_Y + \sqrt\nu \nm{u-U}_{L^2(0,T;Z)} \\
  \leqslant{} &
  C_{u_*,u^*,T,\Omega} \left(
    \nm{y_d}_{Y} +
    (1-4\theta_0)^{-1/2}
  \right) (\theta_0\alpha)^{-1} \tau^{\theta_0\alpha/2}
\end{align*}
for all $ 0 < \theta_0 < 1/4 $. Assuming $ \tau < \exp(-4) $ and inserting $
\theta_0 = 1/4-1/\ln(1/\tau) $ into the above inequality, we then obtain
\begin{equation}
  \label{eq:model-conv}
  \begin{aligned}
    & \nm{
      (y-Y)(T-)
    }_{Y} +
    \sqrt\nu \nm{u-U}_{L^2(0,T;Z)} \\
    \leqslant{} &
    C_{u_*,u^*,T,\Omega} \, \alpha^{-1} \left(
      \nm{y_d}_{L^2(\Omega)} +
      \sqrt{\ln(1/\tau)}
    \right) \tau^{\alpha/8}.
  \end{aligned}
\end{equation}

\begin{remark}
  For the case $ \alpha=1 $,
  \cite{Lasiecka1980-optim,Lasiecka1980,Lasiecka1983,Lasiecka1984,Lasiecka1986}
  used
  \[
    -\mathcal A \int_0^t E_1(t-s) (-\mathcal A)^{-1}
    \mathcal R_{\theta_0} u(s) \, \mathrm{d}s,
    \quad 0 \leqslant t \leqslant T,
  \]
  as the solution of the state equation of problem \cref{eq:model} with $ u \in
  L^2(0,T;Z) $. It is evident that the above solution is exactly $ \mathcal S_1
  \mathcal R_{\theta_0} u $.
\end{remark}

It remains to prove that $ \mathcal S_\alpha \mathcal R_{\theta_0} u $ is a
sensible solution to the state equation of problem \cref{eq:model} for each $ u
\in L^2(0,T;Z) $. To this end, we first introduce the very weak solution concept
of the state equation, following the idea in \cite{Lions1972}. For any $ 0 <
\alpha <1 $ and $ g \in L^2(0,T;Y) $, there exists a unique $ w \in
{}^0H^\alpha(0,T;Y) \cap L^2(0,T;X) $ such that (cf.~\cite{Li2019SIAM,Luo2019})
\[
  (\D_{T-}^\alpha - \mathcal A) w = g
\]
and
\[
  \nm{w}_{{}^0H^\alpha(0,T;Y)} +
  \nm{w}_{L^2(0,T;X)}
  \leqslant C_\alpha \nm{g}_{L^2(0,T;L^2(\Omega))}.
\]
For $ \alpha=1 $ the above results are standard (cf.~\cite{Evans2010}). Hence,
by the method of transposition (cf.~\cite{Lions1972}), we define the very weak
solution $ y \in L^2(0,T;Y) $ to the state equation of problem \cref{eq:model}
with $ u \in L^2(0,T;Z) $ by that
\[
  \int_0^T \big(
    y, \, (\D_{T-}^\alpha - \mathcal A) \varphi
  \big)_Y \, \mathrm{d}t =
  -\dual{u, \partial_{\bm n} \varphi}_{\partial\Omega \times (0,T)}
\]
for all $ \varphi \in {}^0H^\alpha(0,T;Y) \cap L^2(0,T;X) $.

Then we will prove that, for any $ 0 < \theta_0 < 1/4 $ and $ u \in L^2(0,T;Z)
$, $ \mathcal S_\alpha \mathcal R_{\theta_0} u $ is identical to the very weak
solution to the state equation of problem \cref{eq:model}, and hence the
application of the theory in the previous section to problem \cref{eq:model} is
reasonable.
\begin{lemma} 
  Assume that $ 0 < \alpha \leqslant 1 $ and $ 0 < \theta_0 < 1/4 $. Then $
  \mathcal S_\alpha \mathcal R_{\theta_0} u $ is the very weak solution to the
  state equation of problem \cref{eq:model} for each $ u \in L^2(0,T;Z) $.
\end{lemma}
\begin{proof} 
  We only prove the case $ 0 < \alpha < 1 $, the proof of the case $ \alpha=1 $
  being easier. Assume first that $ u \in C([0,T];Z) $. By \cref{lem:Sg-regu} we
  have
  \[
    (\D_{0+}^\alpha - \mathcal A)
    \mathcal S_\alpha \mathcal R_{\theta_0} u =
    \mathcal R_{\theta_0} u
  \]
  and
  \begin{equation}
    \label{eq:100}
    \D_{0+}^\alpha \mathcal S_\alpha\mathcal R_{\theta_0} u, \,
    \mathcal A \mathcal S_\alpha \mathcal R_{\theta_0} u
    \in C([0,T];X^*).
  \end{equation}
  Hence, for any $ \varphi \in {}^0H^\alpha(0,T;Y) \cap L^2(0,T;X) $ we have
  \begin{equation}
    \label{eq:101}
    \int_0^T \dual{
      (\D_{0+}^\alpha - \mathcal A) \mathcal S_\alpha \mathcal R_{\theta_0} u,
      \varphi
    }_X \, \mathrm{d}t = \int_0^T \dual{
      \mathcal R_{\theta_0} u, \varphi
    }_X \, \mathrm{d}t.
  \end{equation}
  Because \cref{eq:100} implies $ \D_{0+}^\alpha \mathcal S_\alpha \mathcal
  R_{\theta_0} u \in L^2(0,T;X^*) $, by \cite[Lemma 3.4]{Luo2019} we have
  \[
    \mathcal S_\alpha \mathcal R_{\theta_0} u
    \in {}_0H^\alpha(0,T;X^*).
  \]
  Also, \cref{eq:100} and \eqref{eq:R(z,A)X*} imply
  \[
    \mathcal S_\alpha \mathcal R_{\theta_0} u
    \in L^2(0,T;Y).
  \]
  Consequently, by \cref{eq:dual} we have
  \begin{align*}
    \int_0^T \dual{
      \D_{0+}^\alpha \mathcal S_\alpha \mathcal R_{\theta_0} u,
      \varphi
    }_X \, \mathrm{d}t =
    \int_0^T (
    \mathcal S_\alpha \mathcal R_{\theta_0} u, \,
    \D_{T-}^\alpha \varphi
    )_Y \, \mathrm{d}t,
  \end{align*}
  and it is evident by \cref{eq:A-ext} that
  \[
    \int_0^T \dual{
      -\mathcal A \mathcal S_\alpha\mathcal R_{\theta_0} u,
      \varphi
    }_X \, \mathrm{d}t = \int_0^T
    (
    \mathcal S_\alpha\mathcal R_{\theta_0} u,
    -\mathcal A^* \varphi
    )_Y \, \mathrm{d}t.
  \]
  Combining \cref{eq:101} and the above two equations gives
  \begin{align*}
    \int_0^T \big(
      \mathcal S_\alpha \mathcal R_{\theta_0} u, \,
      (\D_{T-}^\alpha - \mathcal A^*) \varphi
    \big)_Y \, \mathrm{d}t =
    \int_0^T \dual{\mathcal R_{\theta_0}u, \varphi}_X \, \mathrm{d}t.
  \end{align*}
  The arbitrariness of $ \varphi \in {}^0H^\alpha(0,T;Y) \cap L^2(0,T;X) $
  proves that $ \mathcal S_\alpha\mathcal R_{\theta_0} u $ is indeed the very
  weak solution. The general case $ u \in L^2(0,T;Z) $ then follows from a
  standard density argument by
  \[
    \mathcal S_\alpha\mathcal R_{\theta_0} \in
    \mathcal L(L^2(0,T;Z), L^2(0,T;Y)),
  \]
  which is a direct consequence of \cref{eq:Sg-l2} and the fact $ \mathcal
  R_{\theta_0} \in \mathcal L(Z,[X^*,Y]_{\theta_0}) $. This completes the proof.
\end{proof}


\section{Numerical results}
\label{sec:numer}
This section performs three numerical experiments in two-dimensional space to
verify the theoretical results. We will use the following settings: $ \Omega :=
(0,1) \times (0,1) $; $ T= 0.1 $; $ X $, $ Y $, $ Z $, $ \mathcal R_{\theta_0} $
and $ U_\text{ad} $ are defined as in \cref{sec:appl}.


\medskip\noindent{\it Experiment 1.} Define
\begin{align*}
  g(t) &:= \begin{cases}
    1 & \text{ if } 0 < t < 2/3T, \\
    3 & \text{ if } 2/3T < t < T,
  \end{cases} \\
  v(x,y) &:= \begin{cases}
    y^{-1/2} & \text{ if } (x,y) \in \{(0,y): 0 < y < 1\}, \\
    0 & \text{ if } (x,y) \in \partial\Omega \setminus
    \{(0,y): 0 < y < 1\}.
  \end{cases}
\end{align*}
To approximate $ \mathcal S_{\alpha}\mathcal R_{\theta_0} (gv) $, we use
discretization \cref{eq:Stau}($0<\alpha<1$) or \cref{eq:Stau-1}($\alpha=1$) in
time and use the usual $ H^1(\Omega) $-conforming $ P1 $-element method in
space. Let $ U^M $ be the corresponding numerical approximation with time step $
\tau=T/2^M $ and spatial mesh size $ h=2^{-9} $. Estimates
\cref{eq:S-Stau-g,eq:R-esti} predict that $ \nm{U^M -
U^{13}}_{L^\infty(0,T;L^2(\Omega)\!)} $ is close to $ O(\tau^{0.125}) $ for $
\alpha=0.5 $ and close to $ O(\tau^{0.25}) $ for $ \alpha=1 $, and this is
confirmed by the numerical results in \cref{tab:ex1}.

\begin{table}[H]
  \footnotesize \setlength{\tabcolsep}{2pt}
  \caption{
    $\nm{\cdot}_{L^\infty(L^2)}$ means the norm $\nm{\cdot}_{L^\infty(0,T;L^2(\Omega))}$
  } \label{tab:ex1}
  \begin{tabular}{cccccc}
    \toprule &
    \multicolumn{2}{c}{$\alpha=0.5$} &&
    \multicolumn{2}{c}{$\alpha=1$} \\
    \cmidrule(r){2-3} \cmidrule(r){5-6}
    $M$      & $\nm{U^M-U^{13}}_{L^\infty(L^2)} $ & Order &&
    $\nm{U^M-U^{13}}_{L^\infty(L^2)}$ & Order \\
    $4$  & 6.93e-1 & --   &  & 4.34e-1 & --   \\
    $6$  & 6.26e-1 & 0.07 &  & 3.23e-1 & 0.21 \\
    $8$  & 5.51e-1 & 0.09 &  & 2.24e-1 & 0.26 \\
    $10$ & 4.96e-1 & 0.08 &  & 1.62e-1 & 0.24 \\
    \bottomrule
  \end{tabular}
\end{table}

\medskip\noindent{\it Experiment 2.} Define
\[
  v(x,y) := x^{-1/2}
  \quad\text{for all } (x,y) \in \Omega.
\]
To approximate $ \mathcal S_\alpha (v\delta_0) $, we use discretization
\cref{eq:Stau}($0<\alpha<1$) or \cref{eq:Stau-1}($\alpha=1$) in time and use the
usual $ H^1(\Omega) $-conforming $ P1 $-element method in space. Let $ U^M $ be
the corresponding numerical approximation with time step $ \tau=T/2^M $ and
spatial mesh size $ h=2^{-9} $. \cref{tab:ex2} illustrates that $ \nm{U^M -
U^{13}}_{L^1(0,T;H_0^1(\Omega))} $ is close to $ O(\tau^{0.25}) $ for $
\alpha=0.5 $ and close to $ O(\tau^{0.5}) $ for $ \alpha=1 $, which agrees well
with estimate \cref{eq:S-Stau-vdelta-l1}.

\begin{table}[H]
  \footnotesize \setlength{\tabcolsep}{2pt}
  \caption{
    $\nm{\cdot}_{L^1(H_0^1)}$ means the norm $\nm{\cdot}_{L^1(0,T;H_0^1(\Omega))}$
  }
  \label{tab:ex2}
  \begin{tabular}{cccccc}
    \toprule &
    \multicolumn{2}{c}{$\alpha=0.5$} &&
    \multicolumn{2}{c}{$\alpha=1$} \\
    \cmidrule(r){2-3} \cmidrule(r){5-6}
    $M$      & $\nm{U^M-U^{13}}_{L^1(H_0^1)} $ & Order &&
    $\nm{U^M-U^{13}}_{L^1(H_0^1)}$ & Order \\
    $4$ & 1.71e-0 & --   &  & 4.24e-1 & --   \\
    $5$ & 1.46e-0 & 0.23 &  & 3.01e-1 & 0.49 \\
    $6$ & 1.23e-0 & 0.25 &  & 2.10e-1 & 0.52 \\
    $7$ & 1.02e-0 & 0.27 &  & 1.45e-1 & 0.54 \\
    \bottomrule
  \end{tabular}
\end{table}

\medskip\noindent{\it Experiment 3.} Let $ \nu:=10 $, $ u_* := 0 $, $ u^* := 20
$ and
\[
  y_d(x,y) := 1 \text{ for all } (x,y) \in \Omega.
\]
To approximate problem \cref{eq:model}, we will use the temporal discretization
in \cref{ssec:discrete} and the $ H^1(\Omega) $-conforming $ P1 $-element method
to discretize the state equation in time and space, respectively; see
\cite{Gong2016} for the implementation details. Let $ U^M $ be the corresponding
numerical solution with time step $ \tau = T/2^M $ and spatial mesh size $
h=2^{-8} $. The numerical results in \cref{tab:ex3} show that $
\nm{U^M-U^{13}}_{L^2(0,T;L^2(\partial\Omega))} $ is close to $ O(\tau^{0.125})
$, which agrees with error estimate \cref{eq:model-conv}.
\begin{table}[H]
  \footnotesize \setlength{\tabcolsep}{2pt}
  \caption{
    Numerical results for Experiment 3 with $ \alpha=1 $
  }
  \label{tab:ex3}
  \begin{tabular}{ccc}
    \toprule
    $M$ &     $\nm{U^M-U^{12}}_{L^2(0,T;L^2(\partial\Omega))}$ & Order \\
    \hline
    $4$ & 2.44e-1 & --   \\
    $5$ & 2.08e-1 & 0.23 \\
    $6$ & 1.88e-1 & 0.15 \\
    $7$ & 1.66e-1 & 0.18 \\
    \bottomrule
  \end{tabular}
\end{table}

\bibliographystyle{plain}

\end{document}